\documentclass[11pt,a4paper]{article}
\usepackage{amsmath,amssymb,amsfonts,amsthm}
\usepackage[pdftex]{graphicx, color}

\newcommand{\E}{{\bf{E}}}
\newcommand{\PP}{{\bf{P}}}

\newtheorem{tm}{Theorem}

\newtheorem{lem}{Lemma}

\newtheorem{rem}{Remark}

\begin{document}

\bibliographystyle{plain}
\parindent=0pt
\centerline{\LARGE \bfseries 
Degree-degree distribution  
in a power law}
\centerline{\LARGE \bfseries random intersection graph with clustering}

\par\vskip 3.5em

\centerline{Mindaugas  Bloznelis}

\par\vskip 3.5em

\centerline{Faculty of mathematics and informatics,}
\centerline{ Vilnius university, 
03225 Vilnius, Lithuania}


\bigskip





\begin{abstract}
The bivariate distribution of degrees  
of adjacent vertices, degree-degree distribution,  
is an important network characteristic 
defining
the 
statistical dependencies between degrees of adjacent vertices.
We show the asymptotic degree-degree distribution  
of a 
sparse inhomogeneous random intersection graph and discuss its relation to the 
clustering and power law properties of the graph.
\par
\end{abstract}

\smallskip
{\bfseries key words:}  
degree-degree distribution,  power law, 
clustering coefficient, random intersection graph, affiliation network.

\par\vskip 2.5em


\par\vskip 2.5em

\section{Introduction}

Correlations between degrees of adjacent vertices influence many 
network properties 
including the component structure, epidemic spreading, random walk performance,  network robustness, 
etc., see 
\cite{AvratchenkovMarkovich2014},
 \cite{BogunaPSvespignani_2003}, \cite{HofstadLitvak2014}, \cite{Newman2002},
\cite{Newman2003} and references therein. The correlations
are defined by the degree-degree 
distribution, i.e., the bivariate distribution of
degrees of endpoints of a randomly chosen  edge. 
In this paper we present an analytic study of the
degree-degree distribution
in a mathematically tractable random graph model of an affiliation network 
possessing tunable power law degree distribution and 
 clustering coefficient.
 Our study is motivated  by 
the interest in
 tracing the relation between
the degree-degree distribution 
and  clustering properties in a power law network. 

\medskip

{\bf Affiliation network and random intersection graph.}
An affiliation network defines adjacency relations between actors by using
an auxiliary set of attributes. Let $V$ denote the set of actors 
(nodes of the network)
and $W$ denote the auxiliary set of attributes. Every actor $v\in V$ is 
prescribed a collection $S_v\subset W$ of attributes and two actors $u, v\in V$
are declared adjacent in the network if they share some common attributes.
For example one may interpret elements of $W$ as weights and declare two actors 
adjacent whenever the total weight of shared attributes is above some threshold 
value. Here we consider the simplest case, where $u,v\in V$ are called adjacent 
whenever they share at least one common attribute, i.e., 
$S_u\cap S_v\not=\emptyset$. Two popular 
examples of real affiliation networks are the 
film actor network, where two actors are declared adjacent if they 
have played in the same movie,
and the collaboration network, where two scientists are declared adjacent if 
they have coauthored a publication.

A plausible model of a large affiliation network is obtained by prescribing the 
collections of  attributes to actors at random.
 In order to model the  heterogeneity of human activity, 
every actor $v_j\in V$ is prescribed a random 
weight $y_j$ reflecting their activity. Similarly, 
a random  weight $x_i$ is prescribed 
to each attribute $w_i\in W$ to model its attractiveness. 
Now an attribute $w_i\in W$ is included in the collection $S_{v_j}$ at random 
and with probability proportional to the attractiveness $x_i$ and activity $y_j$ (cf., \cite{karonski1999},
\cite{NewmanSW2001}).  
In this way we obtain a random graph on the vertex set $V$ 
sometimes called
the inhomogeneous random intersection graph, see  
\cite{BloznelisGJKR2015} and references therein.
Before giving a detailed definition of 
this random graph model we
mention a recent publication \cite{muchnik2013}, which 
argues 
convincingly 
that in some social networks the  'heavy-tailed degree distribution 
is causally determined by 
similarly skewed
distribution of human activity'. The empirical evidence reported in 
\cite{muchnik2013}
 suggests that the inhomogeneous
random intersection graph can be considered as 
a realistic model of a power law  affiliation network.

{\bf  Rigorous model.}
Let $X_1,\dots, X_m, Y_1,\dots, Y_n$ be independent non-negative random 
variables such that
each $X_i$ has the  probability distribution $P_1$ and each $Y_j$ has 
the probability 
distribution
$P_2$. Given realized values $X=\{X_i\}_{i=1}^m$ and $Y=\{Y_j\}_{j=1}^n$
we define the random bipartite graph $H_{X,Y}$ with the bipartition 
$V\cup W$, where $V=\{v_1,\dots, v_n\}$ and  $W=\{w_1,\dots, w_m\}$. 
Every  pair $\{w_i,v_j\}$ is linked in $H_{X,Y}$
with probability
\begin{displaymath}
p_{ij}=\min\{1, \lambda_{ij}\},
\qquad
{\rm{where}}
\quad 
\lambda_{ij}=\frac{X_iY_j}{\sqrt{nm}},
\end{displaymath}
 independently 
of the other  pairs $\{w,v\}\in W\times V$.
For large $n$ and $m$, we typically have $\lambda_{ij}<1$ so that
the probability $p_{ij}=\lambda_{ij}$  is proportional to the ``activity'' $Y_j$ of $v_j$ and the 
``attractiveness'' $X_i$ of $w_i$.
The 
 inhomogeneous random intersection graph  $G=G(P_1,P_2, n,m)$ defines 
the adjacency 
relation on the  vertex 
set $V$: vertices $u,v\in V$ are declared adjacent (denoted $u\sim v$)
whenever
$u$ and $v$ have a common neighbor in $H_{X,Y}$. We call this neighbor a 
witness of the adjacency relation $u\sim v$.

\quad
The random graph $G$ has several features that make it a convenient 
theoretical model of a real complex network.
Firstly, the statistical dependence of   neighboring adjacency relations
in $G$ mimics that of real affiliation networks.
In particular, $G$ admits a tunable clustering coefficient: 
For  $m/n\to\beta\in (0,+\infty)$ as $m,n\to+\infty$, we have, 
see \cite{BloznelisKurauskas2016},
\begin{equation}\label{2014-10-20+1}
 {\bf P}(v_1\sim v_2|v_1\sim v_3,\, v_2\sim v_3)=
\frac{\kappa}{ \kappa+\sqrt{\beta}}
+
o(1).
\end{equation}
 Here  $\kappa:=b_1b_2^{-1}a_3a_2^{-2}$ and $a_i={\bf E} X_1^i$, $b_j={\bf E} Y_1^j$. Secondly, 
an important feature of the model is its ability to produce 
a rich class of (asymptotic) degree distributions including 
power laws. Let
$d(v)$ denote the degree of a vertex $v\in V$ in $G$.
We note that, by the symmetry,  random variables  $d(v_1),\dots, d(v_n)$ have 
the 
same probability distribution. 
The following result
 of 
\cite{BloznelisDamarackas2013}
describes the asymptotic distribution of $d(v_1)$ as $n,m\to+\infty$.

\begin{tm}\label{T1}
 Let $m,n\to\infty$. 

(i) Assume that $m/n\to \beta$ for some $\beta\in (0,+\infty)$.
 Suppose  that ${\bf E} X_1^2<\infty$ and ${\bf E} Y_1<\infty$. 
Then  $d(v_1)$ converges in distribution to the random variable
\begin{equation}\label{d*1}
d_*=\sum_{j=1}^{\Lambda_1}\tau_j, 
\end{equation}
where $\tau_1,\tau_2,\dots$ are independent 
and identically distributed random variables 
independent of the random variable $\Lambda_1$.
They are distributed as follows. For $r=0,1,2,\dots$, we have
\begin{equation}\label{d*1++}
{\bf P}(\tau_1=r)=\frac{r+1}{{\bf E}\Lambda_0}{\bf P}(\Lambda_0=r+1)
\qquad
{\rm{and}}
\qquad
{\bf P}(\Lambda_i=r)={\bf E} \,e^{-\lambda_i}\frac{\lambda_i^r}{r!},
\qquad
i=0,1.
\end{equation}
Here $\lambda_0=X_1b_1\beta^{-1/2}$ and $\lambda_1=Y_1a_1\beta^{1/2}$.

(ii) Assume that $m/n\to+\infty$.
 Suppose that  ${\bf E} X_1^2<\infty$ and ${\bf E} Y_1<\infty$. 
Then  $d(v_1)$ converges in distribution to a random variable
$\Lambda_3$ having the probability distribution
\begin{equation}\label{d*2}
{\bf P}(\Lambda_3=r)={\bf E} e^{-\lambda_3}\frac{\lambda_3^r}{r!},
\qquad
r=0,1,\dots.
\end{equation}
Here $\lambda_3=Y_1a_2b_1$.

(iii) Assume that $m/n\to 0$. 
Suppose that  that ${\bf E} X_1<\infty$. Then ${\bf P}(d(v_1)=0)=1-o(1)$.
\end{tm}

We briefly explain the origin of $\Lambda_i$, $i=0,1,3$.
The random variables $\Lambda_0$ and $\Lambda_1$ are limits 
(in distribution) of
the degrees in the bipartite graph $H_{X,Y}$ of $w_1$ and $v_1$ respectively.
 That is,
  the number of attributes
linked to $v_1$  converges in distribution to
$\Lambda_1$ and the number of vertices linked to $w_1$ 
converges in distribution to $\Lambda_0$. Furthermore, the size-biased 
random variable $\tau_j$ counts the neighbors in $G$ of $v_1$ 
witnessed by an attribute $w_j$, given the event 
that $w_j$ is linked to $v_1$. As for  $\Lambda_3$ it 
refers to the case
where $m/n\to +\infty$. Here the number of attributes linked to $v_1$
 grows to infinity (at the rate
$\Theta(\sqrt{m/n})$), while the number of vertices linked to any given attribute
vanishes (an attribute produces a single link 
with a small probability of order $\Theta(\sqrt{n/m})$ or
has no link at all).  Thus among  many attributes linked to $v_1$
only a few
contribute to the degree $d(v_1)$ by witnessing a single neighbor each.
The number of neighbors converges in distribution to a 
mixed Poisson random variable
$\Lambda_3$.

Using the fact that a Poisson random variable 
is highly concentrated around its mean one can 
show that for a power law distribution 
${\bf P}(\lambda_i>r)\approx c_i\, r^{-\varkappa_i}$, 
with some $c_i,\varkappa_i>0$, we have
${\bf P}(\Lambda_i>r)\approx c_ir^{-\varkappa_i}$, for $i=0,1,3$, see
Lemma \ref{L3}. 
Here  and below 
for real sequences 
$\{t_r\}_{r\ge 1}$ and $\{s_r\}_{r\ge 1}$
we denote $t_r \approx s_r$ whenever $t_r/s_r\to 1$ as $r\to+\infty$. 
Furthermore, the tail ${\bf P}(d^*>r)$ of a randomly stopped sum $d_*$ 
is as heavy as the heavier one of $Y_1$ and $\tau_1$, 
see, e.g., \cite{AleskevicieneLeipusSiaulys}. Hence, choosing a power law 
weights $X$ and $Y$ we obtain a power law asymptotic degree distributions, 
namely, the distributions of $d^*$ and $\Lambda_3$.

In what follows we will focus on local probabilities.
Given $c>0$ and $\varkappa>1$, we say  that 
 a non-negative random variable $Z$ 
has the power law property 
${\cal P}_{c,\varkappa}$ (denoted $Z\in {\cal P}_{c,\varkappa}$) if
either $Z$ is integer valued and satisfies ${\bf P}(Z=r)\approx cr^{-\varkappa}$,
or $Z$  is absolutely continuous with density $f_Z$ 
satisfying $f_Z(t)=(c+o(1))t^{-\varkappa}$ as $t\to+\infty$.

\begin{rem}\label{R0}
Let $c, x>0$.  Let $r\to+\infty$.

(i) Let $a>0$ and $\varkappa>3$. Assume that $\E e^{aY_1}<\infty$  
and  $X_1\in {\cal P}_{c,\varkappa}$.
 Then 
$\PP(d_*=r)\approx 
cb_1^{\varkappa-1}\beta^{(3-\varkappa)/2}r^{1-\varkappa}$.

(ii) Let $\varkappa>2$. Assume that  
 $Y_1\in{\cal P}_{c,\varkappa}$ and $\PP(X_1=x)=1$.
 Then 
$\PP(d_*=r)\approx c(x^2b_1)^{\varkappa-1}r^{-\varkappa}$.
\end{rem}

To show (i) we exploit the  power law properties 
   of the local probabilities of 
randomly stopped sums, like $d_*$, in the  case where the 
summands are heavy tailed
and the  number of summands has a light tail 
(see, e.g., Theorem 4.30 of \cite{Foss}).
Unfortunately we are not aware of rigorous 
results establishing  power law properties of the 
local probabilities of randomly stopped sums, like $d_*$, 
in the case where 
 the {\it number of 
summands is heavy tailed}.

\medskip

{\bf Degree-degree distribution.} 
We are interested in the bivariate distribution
of degrees of adjacent vertices. Denote $d_1=d(v_1)$, $d_2=d(v_2)$ and let 
\begin{equation}\label{2014-11-06+2}
 p(k_1,k_2)=\PP(d_1=k_1+1, d_2=k_2+1\,|\,v_1\sim v_2), 
\qquad
k_1,k_2=0,1,\dots,
\end{equation}
denote the probabilities defining the conditional 
bivariate distribution of the {\it ordered} 
pair $(d_1,d_2)$, given the event
that vertices $v_1$ and $v_2$ are adjacent. Let $(u^*,v^*)$ 
be an ordered pair of distinct vertices chosen uniformly at random from $V$.
For the probability distribution of graph 
$G$ is invariant 
under permutations of its vertices, we have
\begin{displaymath}
p(k_1,k_2)=p(k_2,k_1)=\PP\bigl(d(u^*)=k_1+1,d(v^*)=k_2+1|u^*\sim v^*\bigr).
\end{displaymath}

In Theorem \ref{T2} below we 
show a first order asymptotics of $p(k_1,k_2)$ as $n,m\to+\infty$. 
Before formulating the theorem we introduce some notation.
We remark that 
 $d_*$ defined by (\ref{d*1}) depends on $Y_1$. 
By conditioning on the event $\{Y_1=y\}$ we obtain another random 
variable, denoted $d^*_{y}$,
which has the compound Poisson distribution
\begin{displaymath}
 \PP(d^*_y=r)=\PP(d_*=r|Y_1=y)=\PP\Bigl(\sum_{j=1}^N\tau_j=r\Bigr),
\qquad
r=0,1,\dots.
\end{displaymath}
Here $N=N_y$ denotes a Poisson random variable which is independent
of the iid sequence $\{\tau_j\}_{j\ge 1}$ and has mean $\E N_y=ya_1\beta^{1/2}$, 
$y\ge 0$. 
Given integers $k_1,k_2,r\ge 0$, denote
\begin{eqnarray}\label{refereeformula}
&&
 q_r=\E \bigl(Y_1\PP(d^*_{Y_1}=r|Y_1)\bigr)
=
\E \bigl(Y_1\PP(d_*=r|Y_1)\bigr), 
\\
\nonumber
&&
p_{\beta}(k_1,k_2)=\frac{\beta}{b_1^4a_2}
\sum_{r=0}^{k_1\wedge k_2}
(r+1)(r+2)
\PP(\Lambda_0=r+2)
q_{k_1-r}q_{k_2-r},
\\
\nonumber
&&
{\tilde p}(r)
=
r\PP(\Lambda_3=r)\bigl(\E \Lambda_3\bigr)^{-1},
\qquad
p_{\infty}(k_1,k_2)
=
{\tilde p}(k_1+1){\tilde p}(k_2+1).
\end{eqnarray}
Our main result is the following theorem.

\begin{tm}
  \label{T2}
 Let $m,n\to\infty$. Suppose that  $\E X_1^2<\infty$ and $\E Y_1<\infty$.

(i) Assume that $m/n\to \beta$ for some $\beta\in (0,+\infty)$. 
Then  for every  $k_1,k_2\ge 0$ we have
\begin{equation}\label{2014-10-05+4}
 p(k_1,k_2)
=
p_{\beta}(k_1,k_2)+o(1),
\end{equation}

(ii) Assume that $m/n\to+\infty$.
 Then  for every  $k_1,k_2\ge 0$ we have
\begin{equation}\label{2014-10-09+1}
 p(k_1,k_2)
=
p_{\infty}(k_1,k_2)+o(1).
\end{equation}
\end{tm}

We note that the moment conditions $\E X_1^2<\infty$ and $\E Y_1<\infty$ of 
Theorem
\ref{T2}
are the minimal ones as the numbers 
$a_2=\E X_1^2$ and $b_1=\E Y_1$ enter (implicitly) 
both formulas (\ref{2014-10-05+4}) and (\ref{2014-10-09+1}).

\begin{rem}\label{R1}
In the case where $m/n\to+\infty$, the size biased probability distribution 
$\{{\tilde p}(r)\}_{r\ge 1}$
is the limiting distribution of  $d(v_1)$ 
conditioned on the event  $v_1\sim v_2$, i.e.,
\begin{equation}\label{2014-10-09+1++}
\PP(d(v_1)=r|v_1\sim v_2)
=
 {\tilde p}(r)
+
o(1),
\qquad
r=1,2,\dots.
\end{equation}
\end{rem}


Let us examine how the clustering property (presence of non-vanishing  
clustering coefficient)
affects the structure of the  asymptotic degree-degree distribution.
Firstly we consider the case $m/n\to+\infty$, where the clustering coefficient
vanishes (cf. (\ref{2014-10-20+1})). In this case one may expect that
the statistical dependence between neighboring 
adjacency relations fades away as $m,n\to+\infty$  and  that 
the
degrees of adjacent vertices are 
asymptotically independent.
This is indeed the case as
(\ref{2014-10-09+1}) and (\ref{2014-10-09+1++}) imply that 
\begin{displaymath}
p(k_1,k_2)
=
\PP(d(v_1)=k_1+1,|\,v_1\sim v_2)
\PP(d(v_2)=k_2+1\,|\,v_1\sim v_2)
+o(1).
\end{displaymath}
For  $m/n\to \beta\in (0,+\infty)$  the random graph $G$ admits 
a non-vanishing clustering coefficient. Now, the statistical 
dependence of  neighboring 
adjacency relations persists as $n,m\to+\infty$, and the asymptotic 
bivariate degree-degree distribution  is
not a product of marginal asymptotic degree distributions, see 
(\ref{2014-10-05+4}).
One can show that in this case the degrees of adjacent vertices are 
positively correlated and  $G$ admits a non-vanishing 
positive Newman's assortativity 
coefficient, provided that the vertex degree distribution has a finite third 
moment
(cf. \cite{BloznelisJKR2013}). 
  
 Several examples of  power law degree-degree distributions
are considered in Theorem \ref{C1} below.

\begin{tm}\label{C1} Let $a, \beta, c,x>0$.

(i) Assume that $\E X_1^2<\infty$ and   
 $Y_1\in {\cal P}_{c,\varkappa}$, for some $\varkappa>2$.
 Then for $k_1,k_2\to+\infty$ we have
\begin{equation}\label{2014-11-10+++1}
 p_{\infty}(k_1,k_2)
=
(1+o(1))
c^2 a_2^{2\varkappa-4}b_1^{2\varkappa-6} 
(k_1k_2)^{1-\varkappa}.
\end{equation}

(ii) Assume that $\E e^{aY_1}<\infty$ and 
 $X_1\in {\cal P}_{c,\varkappa}$, for some $\varkappa>3$.
 Let  $k_1,k_2\to+\infty$ so that $k_1\le k_2$. Suppose that
 either $k_2-k_1\to+\infty$ or
$k_2-k_1=k$, for an arbitrary, but  fixed integer $k\ge 0$. Then 
\begin{equation}\label{2014-11-10+++2}
p_{\beta}(k_1,k_2)
=
(1+o(1))\frac{\beta}{b_1^4a_2}
\times 
\begin{cases} 
b_1c^*_1c^*_2k_1^{2-\varkappa}(k_2-k_1)^{1-\varkappa} 
&\text{if \  \  $k_2-k_1\to +\infty$,}
\\
c^*_1c^*_{3,k}k_1^{2-\varkappa} & \text{if \ \  $k_2-k_1=k$.}
\end{cases}
\end{equation}
Here $c^*_1=c(b_1\beta^{-1/2})^{\varkappa-1}$, 
$c^*_2=cb_1^{\varkappa-2}b_2\beta^{(3-\varkappa)/2}$,
and
$c^*_{3,k}=\sum_{i\ge 0}q_{i}q_{k+i}$. 
Furthermore, we have $q_{r}\approx c_2^*r^{1-\varkappa}$. We recall that $q_i$ is 
defined in (\ref{refereeformula}).

(iii) 
Assume that $\PP(X_1=x)=1$ and $Y_1\in {\cal P}_{c,\varkappa}$,
 for some $\varkappa>2$.
For $k_1,k_2\to+\infty$ we have  
\begin{equation}\label{2014-11-10+++2+}
p_{\beta}(k_1,k_2)
=
(1+o(1))c^2x^{4\varkappa-8}b_1^{2\varkappa-6}
(k_1k_2)^{1-\varkappa}.
\end{equation}
\end{tm}

We remark, that example (iii) refers to the clustering regime 
($m/n\to\beta\in(0,+\infty)$), where neighboring adjacency
 relations
 are statistically dependent. 
 Hence $p_{\beta}(\cdot,\cdot)$ is not a product of marginal probabilities.
 An interesting fact is that for $k_1,k_2\to+\infty$ the tail of 
 $p_\beta(k_1,k_2)$ is {\it asymptotically} 
 a product of independent marginals. Here we observe a situation, where heavy tailed
weights of actors $Y_j$ define the power law tails of
 $p_\beta(\cdot,\cdot)$ and outperform the light weights of attributes $X_i$.

Our final remark is about the case  where $m/n\to 0$. 
By  Theorem \ref{T1}, in this case  the edges of a  sparse 
inhomogeneous  random intersection graph 
span 
a
subgraph on
$o(n)$ randomly selected vertices leaving the remaining $(1-o(1))n$ 
vertices
isolated. 
Consequently, the subgraph is relatively dense and we do not expect 
stochastically bounded degrees of endpoints of adjacent vertices.

{\bf  Related work.} 
The influence of degree-degree correlations on the network properties
have been studied by many authors, see,  e.g., \cite{BogunaPSvespignani_2003}, 
\cite{HofstadLitvak2014}, \cite{Newman2002},
\cite{Newman2003} and references therein. The asymptotic 
degree-degree distribution
in a preferential attachment random graph with tunable power law 
degree distribution
was shown in \cite{Grechnikov2011}. Our model and approach are much different.
To our best knowledge the present paper is the first attempt to trace
the relation between
the degree-degree distribution 
and the {\it clustering property} in a  power law  network. 
Connections between
Newman's assortativity coefficient and the clustering coefficient 
in  related random graph models have been discussed in  \cite{BloznelisJKR2013}.

The present paper complements, revises and extends the results of \cite{Bloznelis2015},  presented at the 12th Workshop on Algorithms
and Models for the Web Graph, WAW 2015. In particular, the factor $b_1$ is included in (\ref{2014-11-10+++2}). It was missing in the respective formula (7) of  \cite{Bloznelis2015}.


\section{Proof}
Here we prove Theorems \ref{T2}, \ref{C1} and Remarks \ref{R0}, \ref{R1}.  
Before the proof we collect some notation.
We will assume throughout the proof 
that $m=m(n)\to+\infty$ as $n\to+\infty$. The 
expressions $o(\cdot)$, $O(\cdot)$ will always refer to the case 
where $n\to+\infty$. We use the notation $o_P(\cdot)$ and $O_P(\cdot)$ consistently
with \cite{janson2001}.
Given two real sequences $\{a_n\}_{n\ge 1}$ and $\{b_n\}_{n\ge 1}$ we write
$a_n\simeq b_n$ to denote the fact that $(a_n-b_n)mn=o(1)$.
By ${\mathbb I}_{\cal A}$ we denote the indicator function of an event ${\cal A}$ and
${\bar {\mathbb I}}_{\cal A}=1-{\mathbb I}_{\cal A}={\mathbb I}_{\bar{\cal A}}$ 
denotes the indicator of the complement event ${\bar {\cal A}}$. 

In the proof we often use the following facts: For a random variable
$Z$ and a  sequence of events 
$\{{\cal K}_n\}_{n\ge 1}$
defined on the same probability 
space we have  that $\E |Z|<\infty$ and $\lim_n \PP({\cal K}_n)=0$ imply 
$\E ({\mathbb I}_{{\cal K}_n}Z)=o(1)$; for a sequence of random variables
$\{Z_n\}_{n\ge 1}$ the conditions $\exists C>0$ such that $|Z_n|\le C$ almost 
surely $\forall n$, and $Z_n=o_P(1)$ imply $\E |Z_n|=o(1)$.

We denote by $\PP^{\star}$ and $\E^{\star}$ 
(respectively ${\tilde \PP}$ and ${\tilde \E}$)
the conditional probability and expectation
given $X_1,Y_1,Y_2$ (respectively $X,Y$).
 For random variables $\xi,\zeta$ 
defined on the same probability space as $X,Y$ we denote by 
$d^{\star}_{TV}(\xi,\zeta)$ 
(respectively ${\tilde d}_{TV}(\xi,\zeta)$) 
the total variation distance between the 
conditional distributions of $\xi$ and $\zeta$ given $X_1,Y_1,Y_2$ 
(respectively $X,Y$).
Given a sequence of random variables $\{Z_n\}_{n\ge 1}$, 
defined on the same probability space as $Y_1,Y_2,X_1$,
the notation 
$Z_n=o_{P^{\star}}(1)$ means that, given any realized values $Y_1,Y_2,X_1$ and 
any $\varepsilon>0$, we have $\PP^{\star}(Z_n>\varepsilon)=o(1)$. We note that
for 
$\{Z_n\}_{n\ge 1}$  independent of $Y_1,Y_2, X_1$ we have 
$Z_n=o_{P^{\star}}(1)\Leftrightarrow Z_n=o_P(1)$.


The degree of a vertex $v_i\in V$ is denoted $d_i=d(v_i)$. 
The number of common neighbors of 
$v_i,v_j\in V$ is denoted $d_{ij}$.
For a vertex $v\in V$ and attribute $w\in W$ we denote by 
$\{w\to v\}$ the event that $v$ and $w$ are 
linked in $H$. 
Introduce the events 
\begin{equation}\nonumber
 {\cal A}_i=\{w_i\to v_1, w_i\to v_2\}, 
\qquad
1\le i\le m.
\end{equation}
We write for short  
\begin{displaymath}
{\mathbb I}_{ij}={\mathbb I}_{\{w_i\to v_j\}},
\qquad
{\bar {\mathbb I}}_{ij}=1-{\mathbb I}_{ij},
\qquad
{\bf I}_{ij}={\mathbb I}_{\{\lambda_{ij}\le 1\}},
\qquad
{\bar{\bf I}}_{ij}=1-{\bf I}_{ij}. 
\end{displaymath}
Let 
${\bf L}=(L_0,L_1,L_2)$ denote the random vector with marginal random variables
\begin{displaymath}
L_0=u_1,
\qquad
L_1=\sum_{2\le i\le m}{\mathbb I}_{i1}u_i,
\qquad
L_2=\sum_{2\le i\le m}{\mathbb I}_{i2}u_i,
\qquad
u_i=\sum_{3\le j\le n}{\mathbb I}_{ij},
\quad
1\le i\le m.
\end{displaymath}
Let  $\Lambda_i$, $0\le i\le 4$  denote  random variables
having  mixed Poisson distributions
\begin{equation}
 \label{2014-10-21+1}
\PP(\Lambda_i=r)=\E (e^{-\lambda_i}\lambda_i^r/r!),
\qquad
r=0,1,2,\dots,
\end{equation}
where 
\begin{displaymath}
 \lambda_0=X_1b_1\beta^{-1/2},
\quad
\lambda_1=Y_1a_1\beta^{1/2},
\quad
\lambda_2=Y_2a_1\beta^{1/2},
\quad
\lambda_3=Y_1a_2b_1,
\quad
\lambda_4=Y_2a_2b_1
\end{displaymath}
are random variables.
We assume that conditionally, given $Y_1,Y_2,X_1$, the random variables 
$\Lambda_i$, $0\le i\le 4$ 
are independent. 
Define random variables $d^*_{Y_1}=\sum_{i=1}^{\Lambda_1}\tau_i$ 
and $d^*_{Y_2}=\sum_{i=1}^{\Lambda_2}\tau'_i$. Here $\tau_i, \tau'_i, i\ge 1$ are 
independent and identically distributed random variables, which are independent 
of $Y_1,Y_2, X_1$   and have distribution (\ref{d*1++}).
Define the events
\begin{eqnarray}\nonumber
&&
{\cal U}_{k_1,k_2}=\{d_1=k_1+1, d_2=k_2+1\},
\qquad
{\cal U}_{r, r_1,r_2}^*=\{{\bf L}=(r, r_1,r_2)\},
\\
\nonumber
&&
{\cal U}_{r, r_1,r_2}^{**}=\{\Lambda_0=r, d^*_{Y_1}= r_1, d^*_{Y_2}=r_2\},
\qquad
{\cal U}_{r_1,r_2}^{***}=\{\Lambda_3=r_1, \Lambda_4=r_2\}.
\end{eqnarray}
Define random variables 
${\hat a}_k=m^{-1}\sum_{2\le i\le m}X_i^k$
and
${\hat b}_k=n^{-1}\sum_{3\le j\le n}Y_j^k$.


\begin{proof}[Proof of Theorem \ref{T2}]

{\it Proof of (i).} The intuition behind formula (\ref{2014-10-05+4}) is that with a high probability
the adjacency relation $v_1\sim v_2$ 
as well as all common neighbors of $v_1$ and $v_2$ are witnessed, 
 by the same common attribute
(all attributes having equal chances).
 Furthermore, conditionally on the event
that this attribute is $w_1$, and given $Y_1,Y_2, X_1$, 
we have that
the random variables $d_{12}$, $d_1-1-d_{12}=:d_1'$,
$d_2-1-d_{12}=:d_2'$  are asymptotically independent.
We note that $d'_1$ and
$d'_2$ count individual (not shared) neighbors of 
$v_1$ and $v_2$. The individual neighbors of $v_1$ (of $v_2$) are 
attracted 
by all attributes linked to $v_1$ (to $v_2$), but $w_1$, while
 the common neighbors are attracted by the attribute $w_1$. The 
 (conditional given $Y_1,Y_2,X_1$)
asymptotic 
independence of $d_{12}, d'_1, d'_2$
comes  from the fact that 
these random variables are mainly related to each 
other via 
average characteristics ${\hat a}_1$, ${\hat b}_1$ which are 
asymptotically
constant, by the law of large numbers. Now, using Theorem \ref{T1}
we identify limiting distributions of $d'_1, d'_2$.  Finally, the 
limiting distribution 
for the number of vertices
from $\{v_3,\dots, v_n\}$ attracted by $w_1$ is that  of 
$\Lambda_0$.

We briefly outline the proof.  
In the first step  we show that
\begin{equation}\label{2014-10-14+2}
p(k_1,k_2)
=
\frac{\PP({\cal U}_{k_1,k_2}\cap\{v_1\sim v_2\})}{\PP(v_1\sim v_2)}
=
\frac{\PP({\cal U}_{k_1,k_2}\cap(\cup_{i}{\cal A}_i))}{\PP(\cup_{i}{\cal A}_i))}
=
\frac{nm}{a_2b_1^2}\PP({\cal U}_{k_1,k_2}\cap{\cal A}_1)+o(1).
\end{equation}
Then using the total probability formula we
split 
\begin{equation}\label{2014-10-14+3}
\PP({\cal U}_{k_1,k_2}\cap{\cal A}_1)
=
\sum_{r=0}^{k_1\wedge k_2}\PP({\cal U}_{k_1,k_2}\cap\{d_{12}=r\}\cap {\cal A}_1).
\end{equation}
In the second step  we show that for every $r=0,1,\dots, k_1\wedge k_2$ 
\begin{equation}\label{2014-10-14+4}
 \PP({\cal U}_{k_1,k_2}\cap\{d_{12}=r\}\cap {\cal A}_1)
\simeq
\PP({\cal U}^*_{r, k_1-r,k_2-r}\cap {\cal A}_1).
\end{equation}
In the final step we show that for $r, r_1, r_2=0,1,2\dots$
\begin{equation}\label{2014-10-14+5}
\PP({\cal U}^*_{r, r_1,r_2}\cap {\cal A}_1)
\simeq
\E \bigl(\lambda_{11}\lambda_{12}\PP^{\star}({\cal U}^{**}_{r, r_1,r_2})\bigr).
\end{equation}
Now the simple identity
\begin{displaymath}
 nm\E \bigl(\lambda_{11}\lambda_{12} \PP^{\star}({\cal U}^{**}_{r,r_1,r_2})\bigr)
=
\frac{\beta}{b_1^2}(r+1)(r+2)\PP(\Lambda_0=r+2)q_{r_1}q_{r_2}
\end{displaymath}
completes the proof of  (\ref{2014-10-05+4}).
Finally, we remark that in order to prove the result under the minimal 
moment conditions $\E Y_1<\infty$ and $\E X^2_1<\infty$, we invoke, 
when necessary, a 
truncation argument, which makes our presentation some more involved.

{\it {Step 1.}} Here we prove (\ref{2014-10-14+2}). To this aim we show that
\begin{eqnarray}\nonumber
&&
 \PP({\cal U}_{k_1,k_2}\cap \{v_1\sim v_2\})
=
m\PP({\cal U}_{k_1,k_2}\cap{\cal A}_1)
+
o(n^{-1}),
\\
\label{2014-10-15+1}
&&
\PP(\{v_1\sim v_2\})
=
 n^{-1}a_2b_1^2
+
o(n^{-1}).
\end{eqnarray}
Since 
 $\PP({\cal A}_i)\le\E \lambda_{i1}\lambda_{i2}=a_2b_1^2(nm)^{-1}$, 
relations (\ref{2014-10-15+1}) imply (\ref{2014-10-14+2}).
We only prove the second relation. The proof of the first one is much the same.
Fix a small $0<\delta<1$ and introduce truncation events
\begin{displaymath}
 {\cal H}_i=\{Y_i<\delta\sqrt{n}\},
\qquad
 i=1,2.
\end{displaymath}
 We have
\begin{eqnarray}\label{2014-10-15+1+1}
 \PP(\{v_1\sim v_2\})
&=&
\PP(\{v_1\sim v_2\}\cap {\cal H}_1\cap {\cal H}_2)
+
\PP(\{v_1\sim v_2\}\cap {\bar {\cal H}}_1\cap {\cal H}_2)
\\
\nonumber
&+&
\PP(\{v_1\sim v_2\}\cap {\cal H}_1\cap {\bar {\cal H}}_2)
+
\PP(\{v_1\sim v_2\}\cap {\bar {\cal H}}_1\cap {\bar {\cal H}}_2)
\\\
\nonumber
&
=:
&
p'_1+p'_2+p'_3+p'_4.
\end{eqnarray}
Now we evaluate $p'_1$ and construct upper bounds for $p'_i$, $i=2,3,4$.
Using the independence of $Y_1,Y_2$ and Markov's inequality we obtain
\begin{equation}\label{2014-10-15+2}
p'_4 \le
\PP({\bar {\cal H}}_1\cap {\bar {\cal H}}_2)
\le \delta^{-2}n^{-1}(\E(Y_1{\bar {\mathbb I}}_{{\cal H}_1}))^2.
\end{equation}
Next, 
using the identity $\{v_1\sim v_2\}=\cup_{1\le i\le m}{\cal A}_i$ and inequality
\begin{displaymath}
 \PP({\cal A}_1\cap {\bar {\cal H}}_1)
=
\E(\PP^{\star}({\cal A}_1){\mathbb I}_{{\bar {\cal H}}_1})
\le 
\E(\lambda_{11}\lambda_{21}{\bar {\cal H}}_1)
\end{displaymath}
 we obtain
\begin{eqnarray}\label{2014-10-15+3}
p'_2
\le
\sum_{1\le i\le m}
\PP({\cal A}_i\cap {\bar {\cal H}}_1)
=
m
\PP({\cal A}_1\cap {\bar {\cal H}}_1)
\le
\frac{a_2b_1}{n}\E Y_1{\bar {\mathbb I}}_{{\cal H}_1}.
\end{eqnarray}
Clearly, (\ref{2014-10-15+3}) extends to $p'_3$. An upper bound on $p'_1$ is 
obtained in a similar way,
\begin{equation}\label{2014-10-15+4}
 p'_1\le \PP(\{v_1\sim v_2\})\le \sum_{1\le i\le m}\PP({\cal A}_i)
\le \sum_{1\le i\le m}\E \lambda_{1i}\lambda_{2i}
=
n^{-1}a_2b_1^2.
\end{equation}
To get a lower bound we invoke inclusion-exclusion. We have
\begin{eqnarray}\label{2014-10-15+5}
 p'_1
\ge
\sum_{1\le i\le m}\PP({\cal A}_i\cap {\cal H}_1\cap{\cal H}_2)
-
\sum_{1\le i<j\le m}\PP({\cal A}_i\cap{\cal A}_j\cap {\cal H}_1\cap{\cal H}_2)
=:mp'-{\tbinom{m}{2}}p'',
\end{eqnarray}
where $p'=\PP({\cal A}_1\cap {\cal H}_1\cap{\cal H}_2)$ 
and 
$p''=\PP({\cal A}_1\cap{\cal A}_2\cap {\cal H}_1\cap{\cal H}_2)$.
In order to evaluate 
$p'=\E\bigl(\PP^{\star}({\cal A}_1){\mathbb I}_{{\cal H}_1\cap{\cal H}_2}\bigr)$
we invoke the inequalities
\begin{equation}\label{2014-10-15+6}
 (1-{\bar {\bf I}}_{11}-{\bar {\bf I}}_{12})\lambda_{11}\lambda_{12}
\le 
{\bf I}_{11}{\bf I}_{12}\lambda_{11}\lambda_{12}
=
{\bf I}_{11}{\bf I}_{12}\PP^{\star}({\cal A}_1)
\le 
\PP^{\star}({\cal A}_1)
\le
\lambda_{11}\lambda_{12}.
\end{equation}
We obtain
\begin{equation}\label{2014-10-15+7}
p'
\ge
\E
\bigl(
\PP^{\star}({\cal A}_1)
{\mathbb I}_{{\cal H}_1\cap{\cal H}_2}
{\bf I}_{11}{\bf I}_{12}
\bigr)
\ge 
(nm)^{-1}(a_2b_1^2-R),
\end{equation}
where
$R=\E \bigl(
Y_1Y_2X_1^2
({\bar {\bf I}}_{11}+{\bar {\bf I}}_{12}
+
{\bar{\mathbb I}}_{{\cal H}_1\cap{\cal H}_2})
\bigr)
$.
The next upper bound for $p''$ is simple
\begin{equation}\label{2014-10-15+8}
 p''
=
\E \bigl(
\PP^{\star}({\cal A}_1)\PP^{\star}({\cal A}_2)
{\mathbb I}_{{\cal H}_1\cap{\cal H}_2}
\bigr)
\le
\E(\lambda_{11}\lambda_{12}\lambda_{21}\lambda_{22}
{\mathbb I}_{{\cal H}_1\cap{\cal H}_2})
\le 
\delta^2n^{-1}m^{-2}\E (Y_1Y_2X_1^2X_2^2).
\end{equation}
Collecting (\ref{2014-10-15+7}), (\ref{2014-10-15+8}) in (\ref{2014-10-15+5}) 
we obtain  a lower bound for $p_1'$. Combining this lower bound with
 (\ref{2014-10-15+4})  we obtain 
\begin{equation}\label{2014-10-15+9}
n^{-1}(a_2b_1^2-R-\delta^2a_2^2b_1^2)\le p_1'\le n^{-1}a_2b_1^2. 
\end{equation}
Finally, we choose $\delta=\delta_n$ converging to $0$ slowly enough so that
$\delta\sqrt{n}\to+\infty$ 
and 
$\delta^{-1}\E(Y_1{\bar {\mathbb I}}_{H_1})\to 0$.
We obtain from (\ref{2014-10-15+2}), (\ref{2014-10-15+3}) that 
$p'_2,p'_3,p'_4=o(n^{-1})$ and (\ref{2014-10-15+9}) implies that
$p'_1=n^{-1}a_2b_1^2+o(n^{-1})$. Hence (\ref{2014-10-15+1}) follows from 
(\ref{2014-10-15+1+1}).


{\it Step 2}. Here we prove (\ref{2014-10-14+4}). 
Let us note that event ${\cal A}_1$ implies that $L_0$ counts 
common neighbors of $v_1$ and $v_2$ witnessed by $w_1$. 
In the case where $L_0\not= d_{12}$  there should be a 
common neighbor of $v_1, v_2$  witnessed
by some attribute $w_i$ other than $w_1$ or witnessed by two different attributes, 
say $w_{i_1}, w_{i_2}\in W\setminus\{w_1\}$.
We introduce related events
\begin{eqnarray}\nonumber
&&
 {\cal B}_{00}
=
\{
 {\mathbb I}_{i1}{\mathbb I}_{i2}=1
\ 
{\text{for some}}
\
2\le i\le m\},
\\
\nonumber
&&
{\cal B}_{01}
=
\{
 {\mathbb I}_{i_11}{\mathbb I}_{i_1j}{\mathbb I}_{i_22}{\mathbb I}_{i_2j}=1
\
{\text{for some}}
\
3\le j\le n
\
{\text{and}}
\
2\le i_1\not=i_2\le m
\}
\end{eqnarray}
and observe that on the event 
${\cal A}_1\cap {\overline{({\cal B}_{00}\cup{\cal B}_{01})}}$ 
we have $L_0=d_{12}$. Next, assuming that events ${\cal A}_1$
and $L_0=d_{12}$ hold we consider $L_k$ for $k=1,2$.
The sum $L_k$ counts pairs $w_i\to v_j$, for $i\ge 1$, $j\ge 3$,
such that
$v_j$ is a neighbor of $v_k$ witnessed by 
$w_i$. 
Generally, we have   $L_k\ge d_k-1-d_{12}$. This inequality is strict 
if some common neighbor of $v_1, v_2$ witnessed by $w_1$ is also witnessed 
as a neighbor of $v_k$ by some other attribute $w_i$, $i\ge 2$. The inequality
is also strict in the case where some neighbor of $v_k$ is witnessed by two or more
distinct attributes.  If we rule out both of these possibilities, 
we have the equality
$L_k=d_k-1-d_{12}$.
 We introduce the corresponding undesired events
\begin{eqnarray}
\nonumber
&&
 {\cal B}_{k0}
=
\{
{\mathbb I}_{ik}{\mathbb I}_{ij}{\mathbb I}_{1j}=1
\
{\text{for some}}
\
3\le j\le n
\
{\text{and}}
\
2\le i\le m
\},
\\
\nonumber
&&
{\cal B}_{k1}
=
\{
{\mathbb I}_{i_1k}{\mathbb I}_{i_1j}{\mathbb I}_{i_2k}{\mathbb I}_{i_2j}=1
\
{\text{for some}}
\
3\le j\le n
\
{\text{and}}
\
2\le i_1<i_2\le m
\}.
\end{eqnarray}
Finally, we conclude that if ${\cal A}_1$ holds and at 
least one of the following three relations fails
\begin{displaymath}
L_0=d_{12},
\quad
L_1+L_0=d_1-1,
\quad
L_2+L_0=d_2-1
\end{displaymath}
then at least one of the events ${\cal B}_l, {\cal B}_{lj}$, 
$0\le l\le 2$, $0\le j\le 1$,
occurs. Hence, we have
\begin{eqnarray}\nonumber
 \bigl|\PP({\cal U}_{k_1,k_2}\cap \{d_{12}=r\}\cap{\cal A}_1)
-
\PP({\cal U}^*_{r,k_1-r,k_2-r}\cap{\cal A}_1)\bigr|
\le
\sum_{0\le l\le 2}\sum_{0\le j\le 1}\PP({\cal A}_1\cap {\cal B}_{lj}).
\end{eqnarray}
Next we prove that the quantity on the right is $o((nm)^{-1})$.
For this purpose we estimate
\begin{eqnarray}\label{2014-10-17+1}
 \PP({\cal A}_1\cap{\cal B}_{lj})
=
\E \bigl(\PP^{\star}({\cal A}_1)\PP^{\star}({\cal B}_{lj})\bigr)
\le
(mn)^{-1}\E (Y_1Y_2X_1^2\PP^{\star}({\cal B}_{lj}))
\end{eqnarray}
and show 
that  for any realized values $Y_1,Y_2, X_1$ we have 
\begin{equation}\label{2014-10-17+2}
 \PP^{\star}({\cal B}_{lj})=o(1),
\qquad
{\text{for}}
\qquad
0\le l\le 2,
\quad
0\le j\le 1.
\end{equation}
 Since $\PP^*(\cdot)\le 1$ and $\E Y_1Y_2X_1^2<\infty$,
 Lebesgue's dominated convergence theorem then implies the desired bound 
 $\E (Y_1Y_2X_1^2\PP^{\star}({\cal B}_{lj}))=o(1)$.

Let us show (\ref{2014-10-17+2}). For this purpose we write events ${\cal B}_{lj}$ 
in the form ${\cal B}_{lj}=\{S_{lj}\ge 1\}$, where 
\begin{eqnarray}\nonumber
&&
 S_{00}=
\sum_{2\le i\le m}
{\mathbb I}_{i1}
{\mathbb I}_{i2},
\qquad
\qquad
\qquad
S_{01}=
\sum_{3\le j\le n}
\sum_{2\le i_1\not= i_2\le m}
{\mathbb I}_{i_11}
{\mathbb I}_{i_1j}
{\mathbb I}_{i_22}
{\mathbb I}_{i_2j},
\nonumber
\\
&&
S_{k0}=
\sum_{3\le j\le n}\sum_{2\le i\le m}
{\mathbb I}_{ik}
{\mathbb I}_{ij}
{\mathbb I}_{1j},
\qquad
\
\
S_{k1}=
\sum_{3\le j\le n}\sum_{2\le i_1<i_2\le m}
{\mathbb I}_{i_1k}
{\mathbb I}_{i_1j}
{\mathbb I}_{i_2k}
{\mathbb I}_{i_2j},
\end{eqnarray}
and  apply the inequality
\begin{equation}\label{2014-10-17+3}
 \PP^{\star}(S_{lj}\ge 1)
\le 
\varepsilon+\PP^{\star}({\tilde \E}S_{lj}\ge \varepsilon),
\qquad
\forall
\varepsilon>0.
\end{equation}
Let us briefly explain (\ref{2014-10-17+3}). We apply the trivial inequality
$Z\le \varepsilon+Z{\mathbb I}_{\{Z>\varepsilon\}}$ to the random variable
$Z={\tilde\PP}(S_{lj}\ge 1)$ and obtain 
\begin{equation}\label{2014-10-17+4}
{\tilde\PP}(S_{lj}\ge 1)
\le 
\varepsilon+{\mathbb I}_{\{{\tilde\PP}(S_{lj}\ge 1)>\varepsilon\}}
\le
\varepsilon+{\mathbb I}_{\{{\tilde \E}S_{lj}>\varepsilon\}}.
\end{equation}
The last inequality follows by Markov's inequality,
 ${\tilde\PP}(S_{lj}\ge 1)\le {\tilde \E}S_{lj}$. Taking $\E^{\star}$-expected
values in (\ref{2014-10-17+4})
 we arrive to  (\ref{2014-10-17+3}) since
$\PP^{\star}(S_{lj}\ge 1)=\E^{\star}\bigl({\tilde \PP}(S_{lj}\ge 1)\bigr)$. 

\quad Now,  (\ref{2014-10-17+2}) follows from
(\ref{2014-10-17+3}) and the  fact that ${\tilde \E}S_{lj}=o_{P^{\star}}(1)$
for any realized values $Y_1,Y_2,X_1$.
To show that the latter bound holds true we estimate 
\begin{eqnarray}\nonumber
&&
{\tilde \E}S_{00}
\le 
\sum_{2\le i\le m}\lambda_{i1}\lambda_{i2}
=
Y_1Y_2n^{-1}{\hat a}_2
=
o_{P^{\star}}(1),
\\
\nonumber
&&
{\tilde \E}S_{01}
\le
\sum_{3\le j\le n}
\sum_{2\le i_1\not= i_2\le m}
\lambda_{i_11}\lambda_{i_1j}\lambda_{i_22}\lambda_{i_2j}
\le 
Y_1Y_2({\hat a}_2)^2{\hat b}_2n^{-1}
=
o_{P^{\star}}(1),
\\
\nonumber
&&
{\tilde \E}S_{k0}
\le
\sum_{3\le j\le n}\sum_{2\le i\le m}
{\lambda}_{ik}
{\lambda}_{ij}
{\lambda}_{1j}
\le 
Y_1X_1\beta_n^{-1/2}{\hat a}_2{\hat b}_2n^{-1}
=
o_{P^{\star}}(1),
\\
\nonumber
&&
{\tilde \E}S_{k0}
\le
\sum_{3\le j\le n}\sum_{2\le i_1<i_2\le m}
{\lambda}_{i_1k}
{\lambda}_{i_1j}
{\lambda}_{i_2k}
{\lambda}_{i_2j}
\le
Y_k^2({\hat a}_2)^2{\hat b}_2n^{-1}
=
o_{P^{\star}}(1).
\end{eqnarray}
Here we used the fact that $\E X_1^2<\infty$ implies ${\hat a}_2=O_P(1)$ and
$\E Y_1<\infty$ implies ${\hat b}_2n^{-1}=o_P(1)$.


{\it Step 3}. Here we prove (\ref{2014-10-14+5}). 
Denote
\begin{displaymath}
 \Delta_1=\PP^{\star}({\cal U}^*_{r, r_1,r_2})-
\PP^{\star}({\cal U}^{**}_{r, r_1,r_2}),
\qquad
\Delta_2=\PP^{\star}({\cal A}_1)-\lambda_{11}\lambda_{12}.
\end{displaymath}
From identities
\begin{equation}\nonumber
 \PP({\cal U}^*_{r, r_1,r_2}\cap {\cal A}_1)
=
\E \bigl(\PP^{\star}({\cal U}^*_{r, r_1,r_2})\PP^{\star}({\cal A}_1)\bigr),
\quad
\PP({\cal U}^{**}_{r, r_1,r_2}\cap {\cal A}_1)
=
\E \bigl(\PP^{\star}({\cal U}^{**}_{r,r_1,r_2})\PP^{\star}({\cal A}_1)\bigr)
\end{equation}
we obtain
\begin{equation}\label{2014-10-18+1}
 \PP({\cal U}^*_{r,r_1,r_2}\cap {\cal A}_1)
-
\lambda_{11}\lambda_{12}
\PP({\cal U}^{**}_{r,r_1,r_2})
=
\E \bigl(\Delta_1\PP^{\star}({\cal A}_1)\bigr)
+
\E \bigl(\Delta_2 \PP^{\star}({\cal U}^{**}_{r,r_1,r_2})\bigr). 
\end{equation}
We shall show in Lemma \ref{DTV} below that for any realized values $Y_1,Y_2,X_1$
we have
$\Delta_1=o(1)$. This,
together with the inequality $\PP^{\star}({\cal A}_1)\le \lambda_{11}\lambda_{12}$
implies 
\begin{equation}\label{2014-10-18+2}
\E (\Delta_1 \PP^{\star}({\cal A}_1))
\le 
(nm)^{-1}\E (\Delta_1Y_1Y_2X_1^2)
=
o((nm)^{-1}),
\end{equation}
by Lebesgue's dominated convergence theorem.
Furthermore, (\ref{2014-10-15+6})
implies 
\begin{equation}\label{2014-10-18+3}
 \bigl|
\E \bigl(\Delta_2 \PP^{\star}({\cal U}^{**}_{r,r_1,r_2})\bigr)
\bigr|
\le \E|\Delta_2|
\le 
(nm)^{-1}\E Y_1Y_2X_1^2({\bar {\bf I}}_{11}+{\bar {\bf I}}_{12})
=o((nm)^{-1}).
\end{equation}
Collecting (\ref{2014-10-18+2}) and (\ref{2014-10-18+3}) in 
(\ref{2014-10-18+1}) we obtain  (\ref{2014-10-14+5}).


\medskip

{\it Proof of (ii).} 
The proof (\ref{2014-10-09+1}) 
is similar to that of (\ref{2014-10-05+4}). It makes use the 
observation that
the typical adjacency relation 
is  
witnessed by a single attribute. One difference is that now 
the size of the collection of attributes, prescribed to 
the typical vertex, tends to infinity
as $m/n\to+\infty$ while the number of vertices sharing  any given  
attribute tends to zero. 
As a consequence we obtain that $d_{12}=o_P(1)$. 



The first several steps of the proof are the same as that of (\ref{2014-10-05+4}).
Namely, relations
(\ref{2014-10-14+2}),
(\ref{2014-10-14+3}), (\ref{2014-10-14+4}) hold true  as the argument of their proof 
remains  valid for $m/n\to+\infty$.
Further steps of the proof are a bit different. We show that
\begin{equation}\label{2014-10-28+1} 
\PP({\cal U}^*_{r, k_1-r,k_2-r}\cap {\cal A}_1)
\simeq
0,
\qquad
{\text{for}}
\qquad
r=1,2,\dots, k_1\wedge k_2,
\end{equation}
and  
\begin{equation}\label{2014-10-28+2}
\PP({\cal U}^*_{0, k_1,k_2}\cap {\cal A}_1)
\simeq
\PP(\{L_1=k_1, L_2=k_2\}\cap{\cal A}_1).
\end{equation}
Finally, we show that 
\begin{equation}\label{2014-10-28+3}
\PP(\{L_1=k_1, L_2=k_2\}\cap{\cal A}_1)
\simeq
\E \bigl(\lambda_{11}\lambda_{12}\PP^{\star}({\cal U}^{***}_{k_1,k_2})\bigr).
\end{equation}
Now the simple identity
\begin{displaymath}
 a_2^{-1}b_1^{-2}nm
\E \bigl(\lambda_{11}\lambda_{12} \PP^{\star}({\cal U}^{***}_{k_1,k_2})\bigr)
=
{\tilde p}(k_1+1){\tilde p}(k_2+1)
\end{displaymath}
completes the proof of  (\ref{2014-10-09+1}).
It remains to prove 
 (\ref{2014-10-28+1}), (\ref{2014-10-28+2}), (\ref{2014-10-28+3}).

\quad
Let us prove (\ref{2014-10-28+1}), (\ref{2014-10-28+2}). 
Denote $R'=\PP(\{L_0\ge 1\}\cap {\cal A}_1)$.
Relations (\ref{2014-10-28+1}), (\ref{2014-10-28+2}) follow from  inequalities
\begin{eqnarray}
&&
 \PP({\cal U}^*_{r,k_1-r,k_2-r}\cap{\cal A}_1)\le R',
\qquad
r=1,2,\dots,
\\
&&
0
\le 
\PP(\{L_1=k_1,L_2=k_2\}\cap{\cal A}_1)-\PP({\cal U}^*_{0,k_1,k_2}\cap{\cal A}_1)
\le 
R',
\end{eqnarray}
and the bound 
 $R'=o((nm)^{-1})$.
To prove the latter bound we write
\begin{equation}
 R'
=\E \bigl(\PP^{\star}(L_0\ge 1)\PP^{\star}(A_1)\bigr)
\le
 (nm)^{-1}\E (Y_1Y_2X_1^2\PP^{\star}(L_0\ge 1))=o((nm)^{-1}).
\end{equation}
Here we used
inequalities 
$\PP^{\star}({\cal A}_1)\le \lambda_{11}\lambda_{12}\le (nm)^{-1}Y_1Y_2X_1^2$ 
and the bound  
\linebreak
$\E (Y_1Y_2X_1^2\PP^{\star}(L_0\ge 1))=o(1)$, which follows
by Lebesgue's dominated convergence theorem,  
since 
\begin{displaymath}
\PP^{\star}(L_0 \ge 1)\le \E^{\star}L_0\le\sum_{3\le j\le n}\lambda_{1j}
=
b_1X_1((n-2)/n) \beta_n^{-1/2}=o(1).
\end{displaymath}

\quad 
Let us prove (\ref{2014-10-28+3}). The proof is similar to that of
(\ref{2014-10-14+5}). We write
\begin{equation}\label{2014-10-28+5}
 \PP(\{L_1=k_1,L_2=k_2\}\cap {\cal A}_1)
-
\lambda_{11}\lambda_{12}
\PP({\cal U}^{***}_{k_1,k_2})
=
\E \bigl(\Delta_3\PP^{\star}({\cal A}_1)\bigr)
+
\E \bigl(\Delta_2 \PP^{\star}({\cal U}^{***}_{k_1,k_2})\bigr), 
\end{equation}
where 
\begin{displaymath}
\Delta_2=\PP^{\star}({\cal A}_1)-\lambda_{11}\lambda_{12},
\qquad
\Delta_3=\PP^{\star}(L_1=k_1,L_2=k_2)-\PP^{\star}({\cal U}_{k_1,k_2}^{***})
\end{displaymath}
and show that both summands in the right of (\ref{2014-10-28+5}) are $o((nm)^{-1})$,
cf. proof of 
(\ref{2014-10-14+5})
above. In fact, the only thing that remains to show is 
that for any realized values $Y_1,Y_2,X_1$
we have
$\Delta_3=o(1)$. This is shown in  Lemma \ref{DTV+} below.

\end{proof}



\begin{proof}[Proof of Remark \ref{R0}]
Let us prove (i). 
Lemma \ref{L3} implies 
\begin{equation}\label{2014-11-16+0} 
 \PP(\Lambda_0=r)
\approx 
c^*_1r^{-\varkappa},
\qquad
{\text{where}}
\qquad
c^*_1=c(b_1\beta^{-1/2})^{\varkappa-1}.
\end{equation}
Consequently,
\begin{equation}\label{2014-11-16+1}
 \PP(\tau_1=r)\approx ca_1^{-1}(b_1\beta^{-1/2})^{\varkappa-2}r^{1-\varkappa}.
\end{equation}
From the latter relation  we conclude that the sequence of probabilities 
$\{\PP(\tau_1=r)\}_{r\ge 0}$ is long-tailed and sub-exponential, that is,
it satisfies conditions of Theorem 4.30 of \cite{Foss}. This theorem implies
$\PP(\sum_{1\le i\le \Lambda_1}{\tau_i}=r)
\approx 
\PP(\tau_1=r)\E \Lambda_1$, thus completing the proof.

Let us prove (ii). We observe that $\tau_1$ has Poisson distribution with mean $\lambda_0=xb_1\beta^{-1/2}$.
Hence, given $\Lambda_1$, the random variable $d_*$ has Poisson distribution with mean $\lambda_0\Lambda_1$.
Now statement (iii) of Lemma \ref{L3} implies that $\PP(\Lambda_1=r)\approx c_*r^{-\varkappa}$,
where $c_*=c(a_1\beta^{1/2})^{\varkappa-1}$. Next, we apply statement (iii) of Lemma \ref{L3} once again and obtain
$\PP(d_*=r)\approx c_*\lambda_0^{\varkappa-1}r^{-\varkappa}$.

\end{proof}



\begin{proof}[Proof of Remark \ref{R1}]
The proof is standard.
We present it here
for reader's convenience.
Let $(v_1', v_2')$ denote an ordered pair
of distinct vertices drawn uniformly at random and let $\PP'$  denote the 
conditional probability  
given all the random variables considered, but  $(v_1', v_2')$.
We have
\begin{equation}\label{2014-10-28+++2}
 \PP(d(v_1)=r|v_1\sim v_2)
=
\frac{ \PP(d(v_1)=r,v_1\sim v_2)}{\PP(v_1\sim v_2)}
=
\frac{ \PP(d(v'_1)=r,v'_1\sim v'_2)}{\PP(v_1\sim v_2)}.
\end{equation}
The denominator
is evaluated in (\ref{2014-10-15+1}): We have
\begin{displaymath}
\PP(\{v_1\sim v_2\})
=
 n^{-1}a_2b_1^2
+
o(n^{-1})=n^{-1}\E \Lambda_3+o(n^{-1}).
\end{displaymath}
In the last step we used the simple identities
$\E \Lambda_3=\E \lambda_3=a_2b_1^2$.
In order to evaluate the numerator we 
combine identities
\begin{eqnarray}\nonumber
 &&
\PP'(d(v'_1)=r,v'_1\sim v'_2)
=
\PP'(v'_1\sim v'_2|d(v'_1)=r)\PP'(d(v'_1)=r)
 =
\frac{r}{n-1}\PP'(d(v'_1)=r)
\\
\nonumber
&&
\PP(d(v'_1)=r,v'_1\sim v'_2)=\E\bigl(\PP'(d(v'_1)=r,v'_1\sim v'_2)\bigr).
\end{eqnarray}
We obtain 
$\PP(d(v'_1)=r,v'_1\sim v'_2)=(r/(n-1))\PP(d(v_1)=r)$. 
Hence, by (\ref{2014-10-28+++2}), 
\begin{displaymath}
 \PP(d(v_1)=r|v_1\sim v_2)=r\PP(d(v_1)=r)(\E \Lambda_3)^{-1}+o(1).
\end{displaymath}
Now, the statement (ii) of Theorem \ref{T1} completes the proof
of (\ref{2014-10-09+1++}).
\end{proof}

\begin{proof}[Proof of Theorem \ref{C1}](i) follows from the relation 
$\PP(\Lambda_3=r)\approx c(a_2b_1)^{\varkappa-1}r^{-\varkappa}$, see 
Lemma \ref{L3}.

In the proof of (ii) and (iii) we assume that $k_1\le k_2$. We recall that
$q_i$ is defined in (\ref{refereeformula}) and introduce the notation
\begin{displaymath}
 S_A=\sum_{r\in A}
(r+1)(r+2)
\PP(\Lambda_0=r+2)
q_{k_1-r}q_{k_2-r},
\qquad
A\subset [0,k_1].
\end{displaymath}
\quad
Let us prove  (ii). We observe that $\E (e^{aY_1})<\infty$ implies that
$\E Y_1e^{a'\Lambda_1}<\infty$ for some $a'>0$. Using this observation
and  the fact that  the sequence of probabilities 
$\{\PP(\tau_1=r)\}_{r\ge 0}$ is long-tailed and sub-exponential 
(see (\ref{2014-11-16+1}) and \cite{Foss})
 we show that
\begin{equation}\label{2014-11-16++2}
 \E \bigl(Y_1\PP(d^*_{Y_1}=r|Y_1)\bigl)
\approx \bigl(\E (Y_1\Lambda_1)\bigr)\PP(\tau_1=r). 
\end{equation}
The proof of (\ref{2014-11-16++2}) 
is much the same as that of Theorem 4.30 in \cite{Foss}.
Next, we invoke in (\ref{2014-11-16++2}) 
the identity $\E(Y_1\Lambda_1)=\E (Y_1\lambda_1)=a_1b_2\beta^{1/2}$ and 
(\ref{2014-11-16+1}), and obtain
\begin{equation}\nonumber
 \E \bigl(Y_1\PP(d^*_{Y_1}=r|Y_1)\bigl)
\approx
c^*_2r^{1-\varkappa},
\qquad
{\text{where}}
\qquad
c^*_2=cb_1^{\varkappa-2}b_2\beta^{(3-\varkappa)/2}.
\end{equation}
Hence  we have $q_r\approx c^*_2r^{1-\varkappa}$ and $\PP(\Lambda_0=r)
\approx 
c^*_1r^{-\varkappa}$, see (\ref{2014-11-16+0}).

Now we are ready to prove (\ref{2014-11-10+++2}). 
Let $\varepsilon=\ln (k_1\wedge (k_2-k_1))$ for $k_2-k_1\to+\infty$, and 
 $\varepsilon=\ln k_1$ otherwise. 
Split $S_{[0,k_1]}=S_{A_1}+S_{A_2}+S_{A_3}$,
where
\begin{displaymath}
 A_1=[0,k_1/2],
\qquad
A_2=(k_1/2,k_1-\varepsilon ],
\qquad
A_3=(k_1-\varepsilon, k_1].
\end{displaymath}
In the remaining part of the proof we shall show that $S_{A_1}, S_{A_2}$ 
are negligibly small compared to $S_{A_3}$
and determine the first order asymptotics of $S_{A_3}$ as  $k_1,k_2\to+\infty$.
We have for some ${\bar c}>0$ (independent of $k_1,k_2$)
\begin{eqnarray}
 \nonumber
&&
S_{A_1}
\le 
{\bar c}\sum_{i\in A_1}
\frac{1}{(k_1-i)^{\varkappa-1}}
\frac{1}{(k_2-i)^{\varkappa-1}}
\frac{1}{(1+i)^{\varkappa-2}}
=
O\Bigl(k_1^{4-2\varkappa}k_2^{1-\varkappa}(1+\Delta)\Bigr).
\end{eqnarray}
Here $\Delta=\ln n$ for $\varkappa=3$ and $\Delta=0$ otherwise.
Furthermore, for $k_2-k_1$ bounded we have
\begin{equation}\nonumber
S_{A_2}
\le 
\frac{{\bar c}}{k_1^{\varkappa-2}}\sum_{i\in A_2}
\frac{1}{(k_1-i)^{2\varkappa-2}}
= 
O\Bigl(\varepsilon^{3-2\varkappa} k_1^{2-\varkappa}\Bigr)=o(k_1^{2-\varkappa}).
\end{equation}
For $k_2-k_1\to+\infty$ we have
\begin{equation}\nonumber
S_{A_2}
\le 
\frac{{\bar c}}{k_1^{\varkappa-2}}\sum_{i\in A_2}
\frac{1}{(k_1-i)^{\varkappa-1}}
\frac{1}{(k_2-i)^{\varkappa-1}}
\le
\frac{{\bar c}}{k_1^{\varkappa-2}(k_2-k_1)^{\varkappa-1}}\sum_{i\in A_2}
\frac{1}{(k_1-i)^{\varkappa-1}}.
\end{equation}
Since $\sum_{i\in A_2}
\frac{1}{(k_1-i)^{\varkappa-1}}=O(\varepsilon^{2-\varkappa})=o(1)$ we obtain
$S_{A_2}=o\bigl(k_1^{2-\varkappa}(k_2-k_1)^{1-\varkappa}\bigr)$. 

Finally, using the approximation
$(i+1)(i+2)\PP(\Lambda_0=i+2)\approx c^*_1k_1^{2-\varkappa}$ uniformly in
$i\in A_3$ we obtain for $k_2-k_1\to+\infty$
\begin{equation}\nonumber
S_{A_3}
=
\frac{c^*_1(1+o(1))}{k_1^{\varkappa-2}}\sum_{i\in A_3}q_{k_1-i}q_{k_2-i}
\approx
\frac{c^*_1}{k_1^{\varkappa-2}}\frac{c^*_2}{(k_2-k_1)^{\varkappa -1}}
\sum_{i\in A_3}q_{k_1-i}
\approx 
\frac{c^*_1c^*_2b_1}{k_1^{\varkappa-2}(k_2-k_1)^{\varkappa -1}}.
\end{equation}
Here we used  $\sum_{r\ge 0}q_r=b_1$.
Similarly, in the case where $k_2-k_1=k$ for some fixed $k$ we have
\begin{equation}\nonumber
 S_{A_3}
=
\frac{c^*_1(1+o(1))}{k_1^{\varkappa-2}}\sum_{i\in A_3}q_{k_1-i}q_{k_2-i}
\approx
\frac{c^*_1c^*_{3,k}}{k_1^{\varkappa-2}},
\quad
{\text{where}}
\quad 
c^*_{3,k}=\sum_{i\ge 0}q_{i}q_{k+i}.
\end{equation}

\quad
Let us prove (iii). We observe that $\Lambda_0$ has the 
Poisson distribution with (non-random) mean
$\lambda_0$. Using the identity 
$(r+1)(r+2)\PP(\Lambda_0=r+2)=\lambda_0^2\PP(\Lambda_0=r)$ we write
\begin{eqnarray}\nonumber
&&
S_{[0,k_1]}
=
\lambda_0^2\sum_{r=0}^{k_1}\PP(\Lambda_0=r)q_{k_1-r}q_{k_2-r}
=
\lambda_0^2\E\bigl(q_{k_1-\Lambda_0}q_{k_2-\Lambda_0}\bigr)
=
\lambda_0^2(J_1+J_2),
\\
\nonumber
&&
J_1
=
\E\bigl(q_{k_1-\Lambda_0}q_{k_2-\Lambda_0}\bigr)
{\mathbb I}_{\{\Lambda_0<\sqrt k_1\}},
\qquad
J_2
=
\E\bigl(q_{k_1-\Lambda_0}q_{k_2-\Lambda_0}\bigr)
{\mathbb I}_{\{\sqrt{k_1}\le \Lambda_0\le k_1\}}.
\end{eqnarray}
Next, combining the fast decay of Poisson tail probability
$\PP(\Lambda_0>t)$ as $t\to+\infty$ with the relation, which is shown below, 
\begin{equation}\label{2014-11-22+5}
q_r\approx 
c_0r^{1-\varkappa},
\qquad
c_0=c(x^2b_1)^{\varkappa-2},
\end{equation}
 we estimate 
$J_1=(1+o(1)) c_0^2(k_1k_2)^{1-\varkappa}$ and
$J_2=o((k_1k_2)^{1-\varkappa}$.
We obtain that
$S_{[0,1]}=(1+o(1))\lambda_0^2J_1$.
Now the identity  $p_{\beta}(k_1,k_2)=\beta b_1^{-4}x^{-2}S_{[0,k_1]}$ completes 
the proof
of (\ref{2014-11-10+++2+}).

Let us prove (\ref{2014-11-22+5}).
Since $\tau_1$ has Poisson distribution with mean 
$\lambda_0=xb_1\beta^{-1/2}$, we obtain  
\begin{eqnarray}\nonumber
 \PP\bigl(d^*_{Y_1}=k|Y_1=y\bigr)
&=&
\E\Bigl(\PP\bigl(d^*_{Y_1}=k\bigl|\Lambda_1, Y_1=y)\Bigr|Y_1=y\Bigr)
\\
\nonumber
&=&
\E
\Bigl(
e^{-\lambda_0\Lambda_1}\frac{(\lambda_0\Lambda_1)^k}{k!}
\Bigl|
Y_1=y
\Bigr)
\\
\nonumber
&=&
\sum_{i\ge 0}
e^{-\lambda_0i}\frac{(\lambda_0i)^k}{k!}e^{-By}\frac{(By)^i}{i!}.
\end{eqnarray}
Here we denote $B=x\beta^{1/2}$. After we write the product $y\PP(d^*_y=k)$ 
in the form
\begin{displaymath}
\sum_{i\ge 0}
e^{-\lambda_0i}\frac{(\lambda_0i)^k}{k!}e^{-By}\frac{(By)^{i+1}}{(i+1)!}
\frac{i+1}{B}
=
\E
\Bigl(
e^{-\lambda_0(\Lambda_1-1)}\frac{(\lambda_0(\Lambda_1-1))^k}{k!}
\frac{\Lambda_1}{B}
{\mathbb I}_{\{\Lambda_1\ge 1\}}
\Bigr| Y_1=y
\Bigr) 
\end{displaymath}
we obtain the following expression for the expectation 
$q_k=\E \bigl(Y_1\PP(d^*_{Y_1}=k|Y_1)\bigr)$
\begin{equation}\label{2014-11-22+2}
 q_k=\frac{\PP(\Lambda_1\ge 1)}{B}(I_{1,k}+I_{2,k}),
\quad
I_{1,k}=\E \Bigl(e^{-\lambda_0Z}\frac{(\lambda_0Z)^k}{k!}\Bigr),
\quad
I_{2,k}=\E \Bigl(Ze^{-\lambda_0Z}\frac{(\lambda_0Z)^k}{k!}\Bigr).
\end{equation}
Here $Z$ denotes a random variable with the distribution
\begin{displaymath}
 \PP(Z=r)=\PP(\Lambda_1=r+1)/\PP(\Lambda_1\ge 1),
\qquad
 r=0,1,\dots.
\end{displaymath}
We note that 
\begin{equation}\label{2014-11-22+1}
 \PP(Z=r)\approx c' r^{-\varkappa},
\qquad
{\text{ where}}
\qquad
c'=cB^{\varkappa-1}/\PP(\Lambda_1\ge 1). 
\end{equation}
Indeed, (\ref{2014-11-22+1}) follows from the relation 
$\PP(\Lambda_1=r)\approx cB^{\varkappa-1}r^{-\varkappa}$, 
which is a simply consequence of
the property ${\cal P}_{c,\varkappa}$ of the distribution of $Y_1$, see
Lemma \ref{L3}.
Next, we show that 
\begin{equation}\label{2014-11-22+3}
 I_{1,k}\approx c'\lambda_0^{\varkappa-1}k^{-\varkappa},
\qquad
I_{2,k}\approx c'\lambda_0^{\varkappa-2}
k^{1-\varkappa}.
\end{equation}
The first relation follows from (\ref{2014-11-22+1}), by Lemma \ref{L3}.
The second relation follows from the first one via the simple identity
$I_{2,k}=(k+1)\lambda_0^{-1}I_{1,k+1}$. Finally
invoking (\ref{2014-11-22+3}) in (\ref{2014-11-22+2})
we obtain (\ref{2014-11-22+5}).

\end{proof}



\section{Appendix A}

Here we prove the bounds $\Delta_1=o(1)$ and $\Delta_3=o(1)$, see 
(\ref{2014-10-18+1}) and
(\ref{2014-10-28+5}) above. In the proof  we apply and 
further extend
the approach of \cite{BloznelisDamarackas2013}. An important tool used below
is  the following inequality 
 referred to as Le Cam's lemma, see e.g., \cite{Steele}.
\begin{lem}\label{LeCamLemma} Let $S={\mathbb I}_1+ {\mathbb I}_2+\dots+ {\mathbb I}_n$ be the sum of independent random indicators
with probabilities $\PP({\mathbb I}_i=1)=p_i$. Let $\Lambda$ be Poisson random variable with mean $p_1+\dots+p_n$. The total variation
distance between the distributions $P_S$ of $P_{\Lambda}$ of $S$ and $\Lambda$
\begin{equation}\label{LeCam}
\sup_{A\subset \{0,1,2\dots \}}|\PP(S\in A)-\PP(\Lambda\in A)|= \frac{1}{2} \sum_{k\ge 0}|\PP(S=k)-\PP(\Lambda=k)|
\le
\sum_{i}p_i^2.
\end{equation}
\end{lem}

\begin{lem}
  \label{DTV}
   Assume that conditions of part (i) of Theorem \ref{T2} are satisfied.
Then for any realized values $Y_1,Y_2,X_1$ and any integers $r,r_1,r_2\ge 0$
we have
\begin{equation}\label{2014-10-21++0}
 \PP^{\star}({\cal U}^*_{r, r_1,r_2})-\PP^{\star}({\cal U}^{**}_{r, r_1,r_2})=o(1).
\end{equation}
\end{lem}

\begin{proof}[Proof of Lemma \ref{DTV}] Denote  
\begin{displaymath}
 {\bf L}^{(0)}=(L_0^{(0)}, L_1^{(0)}, L_2^{(0)}):=(L_0,L_1,L_2),
\qquad
{\bf L}^{(4)}=(L_0^{(4)}, L_1^{(4)}, L_2^{(4)}):=(\Lambda_0, d^*_{Y_1}, d^*_{Y_2}).
\end{displaymath}
 In the proof we construct random vectors 
${\bf L}^{(h)}=(L_0^{(h)}, L_1^{(h)}, L_2^{(h)})$, $h=1,2,3$, defined on the same 
probability space as $Y_1,Y_2, X_1$ such that for any realized values $Y_1,Y_2,X_1$ 
we have
\begin{eqnarray}\label{2014-10-21++1}
 &&
d^{\star}_{TV}({\bf L}^{(h)}, {\bf L}^{(h+1)})=o(1),
\qquad
h=0,1,
\\
\label{2014-10-21++2}
&&
\PP^{\star}\bigl({\bf L}^{(h)}=(r,r_1,r_2)\bigr)
-
\PP^{\star}\bigl({\bf L}^{(h+1)}=(r,r_1,r_2)\bigr)
=o(1),
\qquad
h=2, 3,
\end{eqnarray}
for any integers $r,r_1,r_2\ge 0$.
We note that (\ref{2014-10-21++1}), (\ref{2014-10-21++2})
imply (\ref{2014-10-21++0}).

\smallskip

Let us define random vectors ${\bf L}^{(h)}$, $h=1,2,3$. 
Given $X,Y$, let $\{\eta_{ki}\}_{i=2}^m$, $k=1,2$ and  
$\{\xi_{hi}\}_{i=1}^m$, $h=2,3$, be sequences of Poisson random
variables which are (conditionally, given $X,Y$) independent within each sequence
and have mean values
\begin{equation}\nonumber
 {\tilde \E}\eta_{ki}=\lambda_{ki},
\qquad
{\tilde \E}\xi_{2i}=\sum_{3\le j\le n}\lambda_{ij},
{\tilde \E}\xi_{3i}=\frac{b_1}{\sqrt{\beta}}X_i.
\end{equation}
In addition, we assume that (conditionally, given $X,Y$) the sequences 
$\{\eta_{ki}\}_{i=2}^m$, $k=1,2$, are independent and they are independent of the 
sequences $\{\xi_{hi}\}_{i=1}^m$, $h=2,3,4$.
Denote for $k=1,2$ and $h=2,3$
\begin{equation}\nonumber
 L_0^{(1)}=u_1, 
\qquad
L_0^{(h)}=\xi_{h1},
\qquad
L_k^{(1)}=\sum_{2\le i\le m}\eta_{ki}u_i,
\qquad
L_k^{(h)}=\sum_{2\le i\le m}\eta_{ki}\xi_{hi}.
\end{equation}

\quad
Let us prove (\ref{2014-10-21++1}) for $h=0$. 
In the proof we use the following simple inequalities. Let
${\tilde d}^*_{TV}(\zeta,\theta)$ denote the total variation distance between the 
conditional distributions of random variables/vectors $\zeta$ and $\theta$ given 
$X,Y,u_1,\dots, u_m$. 
Then  we have
\begin{equation}\label{2014-10-23+1}
 {\tilde d}_{TV}(\zeta,\theta)
\le 
{\tilde \E}{\tilde d}^*_{TV}(\zeta,\theta),
\qquad
d^{\star}_{TV}(\zeta,\theta)
\le 
\E^{\star}{\tilde d}_{TV}(\zeta,\theta)
\le
\E^{\star}{\tilde d}^*_{TV}(\zeta,\theta).
\end{equation}
Introduce random variables
for $k=1,2$ and $t=1,\dots, m$
\begin{displaymath}
 L_k^{[t]}=\sum_{2\le i\le t}{\mathbb I}_{ik}u_i+\sum_{t+1\le i\le m}\eta_{ki}u_i.
\end{displaymath}
Note that $L_k^{[1]}=L_k^{(1)}$ and $L_k^{[m]}=L_k^{(0)}$. 
We have, 
by the triangle inequality,
\begin{displaymath}
 {\tilde d}^*_{TV}(L_k^{[1]},L_k^{[m]})
\le
\sum_{2\le t\le m} {\tilde d}^*_{TV}(L_k^{[t-1]},L_k^{[t]})
\le 
\sum_{t=2}^m\lambda^2_{tk}.
\end{displaymath}
Here we estimated each summand
\begin{equation}\label{2014-10-22+1}
 {\tilde d}^*_{TV}(L_k^{[t-1]},L_k^{[t]})
\le
{\tilde d}^*_{TV}({\mathbb I}_{tk},\eta_{kt})\le \lambda^2_{tk}.
\end{equation}
In the first inequality of  (\ref{2014-10-22+1}) we use the fact that 
$L_k^{[t-1]}$ and $L_k^{[t]}$ only differ in the $t$-th summand. The second 
inequality is trivial for $\lambda_{tk}\ge 1$ and it follows 
 by Le Cam's inequality, see Lemma \ref{LeCamLemma}, for  $\lambda_{tk}<1$.
 
\quad
Next, we use the fact that given $X,Y,u_1,\dots, u_m$ the random vectors
${\bf L}^{(0)}$ and ${\bf L}^{(1)}$ have conditionally 
independent marginals. In particular, we can apply the triangle inequality 
\begin{displaymath}
 {\tilde d}^*_{TV}({\bf L}^{(0)},{\bf L}^{(1)})
\le 
\sum_{0\le k\le 2}{\tilde d}^*_{TV}(L_k^{(0)},L_k^{(1)})
\le \sum_{2\le t\le m}(\lambda_{t1}^2+\lambda_{t2}^2).
\end{displaymath}
In the last step we use ${\tilde d}^*_{TV}(L_0^{(0)},L_0^{(1)})=0$.
Finally, (\ref{2014-10-23+1}) implies 
\begin{displaymath}
 d^{\star}_{TV}({\bf L}^{(0)},{\bf L}^{(1)})
\le 
\E^{\star} \sum_{2\le t\le m}(\lambda_{t1}^2+\lambda_{t2}^2)
=
\frac{2}{n}\frac{m-2}{m}a_2(Y_1^2+Y_2^2)=o(1).
\end{displaymath}


\quad
Let us prove (\ref{2014-10-21++1}) for $h=1$. We note that conditionally, given
$X,Y$, the random variable $L_0^{(1)}$ is independent of $(L_1^{(1)}, L_2^{(1)})$, 
and  $L_0^{(2)}$ is independent of $(L_1^{(2)}, L_2^{(2)})$. Hence, we have, by 
triangle inequality, 
\begin{displaymath}
 {\tilde d}_{TV}\bigl({\bf L}^{(1)}, {\bf L}^{(2)}\bigr)
\le
{\tilde d}_{TV}(L_0^{(1)},L_0^{(2)})
+
{\tilde d}_{TV}\bigl( (L_1^{(1)}, L_2^{(1)}),   (L_1^{(2)}, L_2^{(2)}) \bigr)
=: Z_1+Z_2.
\end{displaymath}
We shall show that $\E^{\star}Z_i=o(1)$, $i=1,2$. These bounds combined with the second 
inequality of (\ref{2014-10-23+1}) imply 
$d^{\star}_{TV}\bigl({\bf L}^{(1)}, {\bf L}^{(2)}\bigr)=o(1)$.

We firstly estimate $Z_1$. We have 
\begin{equation}\label{2014-10-06+1}
Z_1
=
{\tilde d}_{TV}(u_1,\xi_{21})
\le
\sum_{3\le j\le n}\lambda^2_{1j}
=
\beta_n^{-1}X_1^2{\hat b}_2n^{-1}.
\end{equation}
Indeed, for
$\max_j\lambda_{1j}<1$
this inequality follows 
 by Lemma \ref{LeCamLemma}.
Otherwise the inequality is trivial, since the total variation distance is always
less than or equal to $1$.
Next, we use the fact that $\E Y_1<\infty$ implies ${\hat b}_2n^{-1}=o_P(1)$.
We obtain $Z_1=o_{P^{\star}}(1)$. This bound together with the inequality 
$Z_1\le 1$ implies $\E^{\star}Z_1=o(1)$.

We secondly estimate $Z_2$. Introduce random variables
for $k=1,2$ and $t=1,\dots, m$
\begin{displaymath}
 L_k^{\{t\}}=\sum_{2\le i\le t}\eta_{ki}u_i+\sum_{t+1\le i\le m}\eta_{ki}\xi_{2i}.
\end{displaymath}
Note that $L_k^{\{1\}}=L_k^{(2)}$ and $L_k^{\{m\}}=L_k^{(1)}$. 
We apply  triangle inequality
\begin{equation}\label{2014-10-23+4}
{\tilde d}_{TV}\bigl((L_1^{\{m\}},L_2^{\{m\}}),(L_1^{\{1\}},L_2^{\{1\}})\bigr)
\le
\sum_{t=2}^m
{\tilde d}_{TV}\bigl((L_1^{\{t-1\}},L_2^{\{t-1\}}),(L_1^{\{t\}},L_2^{\{t\}})\bigr)
\end{equation}
and  estimate each summand as follows
\begin{eqnarray}\label{2014-10-5+8}
 {\tilde d}_{TV}\bigl((L_1^{\{t-1\}},L_2^{\{t-1\}}),(L_1^{\{t\}},L_2^{\{t\}})\bigr)
&\le& 
{\tilde d}_{TV}
\bigl((u_t\eta_{t1}, u_t\eta_{t2}),(\xi_{2t}\eta_{t1},\xi_{2t}\eta_{t2})\bigr)
\\
\label{2014-10-5+9}
&\le& 
{\tilde \PP}\bigl((\eta_{t1},\eta_{t2})\not=(0,0)\bigr)
{\tilde d}_{TV}(u_t,\xi_{2t})
\\
\label{2014-10-5+10}
&
\le
&
(\lambda_{t1}+\lambda_{t2})\min\bigl\{1,\, \beta_n^{-1}X_t^2{\hat b}_2n^{-1}\bigr\}.
\end{eqnarray}
In  (\ref{2014-10-5+8}) we use the fact that 
$L_k^{\{t-1\}}$ and $L_k^{\{t\}}$ only differ in the $t$-th summand.
 In (\ref{2014-10-5+9}) we  first estimate the total variation distance
between conditional distributions of $(u_t\eta_{t1}, u_t\eta_{t2})$
 and $(\xi_{2t}\eta_{t1},\xi_{2t}\eta_{t2})$ given $\eta_{t1}, \eta_{t2}, X,Y$
from above by 
${\mathbb I}_{\{(\eta_{t1},\eta_{t2})\not=(0,0)\}}{\tilde d}_{TV}(u_t,\xi_{2t})$
and then take the expected value ${\tilde \E}$.
 In (\ref{2014-10-5+10}) we estimate 
${\tilde \PP}\bigl((\eta_{t1},\eta_{t2})\not=(0,0)\bigr)
\le
\lambda_{t1}+\lambda_{t2}$. This bound is trivial for 
$\lambda_{t1}+\lambda_{t2}\ge 1$. Otherwise it follows from the inequalities
\begin{displaymath}
 {\tilde \PP}\bigl((\eta_{t1},\eta_{t2})\not=(0,0)\bigr)
\le
{\tilde \PP}(\eta_{t1}\not=0)
+
{\tilde \PP}(\eta_{t2}\not=0)
=e^0-e^{-\lambda_{t1}}+e^0-e^{-\lambda_{t2}}
\le 
\lambda_{t1}+\lambda_{t2}.
\end{displaymath}
In(\ref{2014-10-5+10}) we invoke the inequality
${\tilde d}_{TV}(u_t,\xi_{2t})
\le 
\min\bigl\{1,\, \beta_n^{-1}X_t^2{\hat b}_2n^{-1}\bigr\}$, 
cf. (\ref{2014-10-06+1}) above.

\quad
Next, using the identities $L_k^{\{1\}}=L_k^{(2)}$ and $L_k^{\{m\}}=L_k^{(1)}$
we obtain from (\ref{2014-10-23+4}), (\ref{2014-10-5+10}) that
\begin{eqnarray}\label{2014-10-06+2}
Z_2\le
\sum_{t=2}^m(\lambda_{t1}+\lambda_{t2})
\min\bigl\{1,\, \frac{X_t^2{\hat b}_2}{\beta_nn}\bigr\}
=
(Y_1+Y_2)\beta_n^{1/2}
\sum_{t=2}^m\frac{X_t}{m}\bigl\{1,\, \frac{X_t^2{\hat b}_2}{\beta_nn}\bigr\}.
\end{eqnarray}
Finally, we have,
\begin{displaymath}
\E^{\star}Z_2
=
\beta^{1/2}(Y_1+Y_2)
\E^{\star}X_2\bigl\{1,\, X_2^2{\hat b}_2\beta_n^{-1}n^{-1}\bigr\}
=
\beta^{1/2}(Y_1+Y_2)
\E X_2\bigl\{1,\, X_2^2{\hat b}_2\beta_n^{-1}n^{-1}\bigr\}
=
o(1),
\end{displaymath}
since ${\hat b}_2$ is independent of $X_2$ and $\E X_2<\infty$, and 
${\hat b}_2n^{-1}=o_P(1)$.


\quad
Let us prove (\ref{2014-10-21++2}) for $h=2$.
Given $X,Y$, let $\{\xi_i\}_{i=1}^m$, $\{\Delta'_i\}_{i=1}^m$, 
$\{\Delta''_i\}_{i=1}^m$ be independent Poisson random variables, which are 
independent of $\{\eta_{i1}\}_{i=2}^m$, $\{\eta_{i2}\}_{i=2}^m$, and have 
mean values
\begin{equation}\label{2014-10-23++1}
 {\tilde \E}\xi_i=bX_i,
\qquad
{\tilde \E}\Delta'_i=\delta'X_i,
\qquad
{\tilde \E}\Delta''_i=\delta''X_i,
\qquad
1\le i\le m.
\end{equation}
Here $\delta'={\hat b}_1\beta_n^{-1/2}-b$ and $\delta'=b_1\beta^{-1/2}-b$,
and $b=\min\{{\hat b}_1\beta_n^{-1/2}, b_1\beta^{-1/2}\}$.
Define the random vectors ${\bar {\bf L}}^{(2)}$ and 
${\bar {\bf L}}^{(3)}$ in the same way as ${\bf L}^{(2)}$ and 
${\bf L}^{(3)}$ above, but with $\xi_{2i}$ and $\xi_{3i}$ replaced by
${\bar \xi}_{2i}=\xi_i+\Delta'_i$ and 
${\bar \xi}_{3i}=\xi_i+\Delta''_i$ respectively. We note that given $Y_1,Y_2, X_1$
the random vector 
${\bar {\bf L}}^{(2)}$ has the same conditional distribution as 
 ${\bf L}^{(2)}$, and  
${\bar {\bf L}}^{(3)}$ has the same conditional distribution as 
${\bf L}^{(3)}$. Hence in order to prove  (\ref{2014-10-21++2}) for $h=2$ it 
suffices to show that for any realized values $Y_1,Y_2,X_1$ we have
\begin{equation}\label{2014-10-23++2}
\PP^{\star}(|{\bar L}_k^{(2)}-{\bar L}_k^{(3)}|\ge 1)=o(1),
\qquad
k=0,1,2.
\end{equation}
Here ${\bar L}_k^{(i)}$ denotes the $k$-th coordinate of the vector 
${\bar{\bf L}}^{(i)}$.
We derive (\ref{2014-10-23++2}) from the bounds 
${\tilde \E}|{\bar L}_k^{(2)}-{\bar L}_k^{(3)}|=o_{P^{\star}}(1)$
using (\ref{2014-10-17+3}).
To prove these bounds we 
calculate
\begin{eqnarray}\nonumber
&&
{\tilde \E}|{\bar L}_0^{(2)}-{\bar L}_0^{(3)}|
=
(\delta'+\delta'')X_1,
\\
\nonumber
&&
 {\tilde \E}|{\bar L}_k^{(2)}-{\bar L}_k^{(3)}|
=
\sum_{2\le i\le m} {\tilde \E}\eta_{ki}(\delta'+\delta'')X_i
=
\beta_n^{1/2}Y_k{\hat a}_2(\delta'+\delta''),
\qquad
k=1,2
\end{eqnarray}
and use the fact that 
 $\delta'+\delta''=|{\hat b}_1\beta_n^{-1/2}-b_1\beta^{-1/2}|=o_P(1)$.
Indeed, we have
 $\beta_n\to\beta$ and, by the law of large numbers,
${\hat b}_1-b_1=o_P(1)$.


Let us prove (\ref{2014-10-21++2}) for $h=3$. We note that $L_0^{(3)}$ has the same
distribution as $\Lambda_0$. Furthermore, given $Y_1,Y_2,X_1$, the random variable
$L_0^{(3)}$ is conditionally independent of $(L_1^{(3)}, L_2^{(3)})$. Hence, 
it suffices to show that 
$\PP^*((L_1^{(3)},L_2^{(3)})=(r_1,r_2))-\PP^*((L_1^{(4)},L_2^{(4)})=(r_1,r_2))=o(1)$
for any integers $r_1,r_2\ge 0$. For this purpose we prove the convergence
of the conditional characteristic functions 
\begin{equation}\label{2014-10-27+1}
 \E^{\star} e^{{\bf i}tL_1^{(3)}+{\bf i}sL_2^{(3)}}
\to 
\E^{\star} e^{{\bf i}tL_1^{(4)}+{\bf i}sL_2^{(4)}},
\qquad
\forall s,t\in (-\infty,+\infty).
\end{equation}
Here ${\bf i}$ denotes the imaginary unit. 

\quad
In the proof of (\ref{2014-10-27+1}) we exploit the compound Poisson structure of 
the distributions of $L_k^{(h)}$, $h=2,3$, $k=1,2$.
Using the fact that, given $Y_1,Y_2,X_1$, the marginals of $(L_1^{(4)},L_2^{(4)})$ 
are conditionally independent and have compound Poisson distributions, 
we write the  characteristics function in the form
\begin{equation}\label{2014-10-27+2}
\E^{\star} e^{{\bf i}tL_1^{(4)}+{\bf i}sL_2^{(4)}}
=
\bigl(\E^{\star} e^{{\bf i}tL_1^{(4)}}\bigr)
\cdot
\bigl(\E^{\star} e^{{\bf i}sL_2^{(4)}}\bigr)
=
e^{(f_{\tau}(t)-1)\lambda_1}
\cdot
e^{(f_{\tau}(s)-1)\lambda_2}. 
=:
f_1(t)
\cdot
f_2(s).
\end{equation}
Here $f_\tau(t)=\E e^{it\tau}$ denotes the characteristic function of $\tau$. 
Similarly, using the fact that given $X,Y, \xi_{32},\dots,\xi_{3m}$ the marginals
of 
$(L_1^{(3)},L_2^{(3)})$ 
are conditionally independent and have compound Poisson distributions, 
we write the  conditional characteristics function in the form
\begin{equation}\label{2014-10-27+3}
{\hat \E} e^{{\bf i}tL_1^{(3)}+{\bf i}sL_2^{(3)}}
=
\bigl({\hat \E} e^{{\bf i}tL_1^{(3)}}\bigr)
\cdot
\bigl({\hat \E} e^{{\bf i}sL_2^{(3)}}\bigr)
=
e^{({\hat f}_{\tau}(t)-1){\hat\lambda}_1}
\cdot
e^{({\hat f}_{\tau}(s)-1){\hat \lambda}_2}
=:
g_1(t)
\cdot
g_2(s).
\end{equation}
Here ${\hat \E}$ denotes the conditional expectation given 
$X,Y, \xi_{32},\dots,\xi_{3m}$. Furthermore, we denote 
${\hat \lambda}_k=\sum_{2\le j\le m}\lambda_{jk}$  and 
\begin{displaymath}
 {\hat f}_{\tau}(t)=\sum_{r\ge 0}e^{itr}{\hat p}_r, 
\qquad
{\text{where}}
\qquad
{\hat p}_r
=
{\hat \lambda}_1^{-1}\sum_{2\le j\le m}\lambda_{j1}{\mathbb I}_{\{\xi_{3j}=r\}}.
%
\end{displaymath}
We observe that ${\hat \lambda}_k=Y_k{\hat a}_1\sqrt{\beta_n}$ and 
each  ratio $\lambda_{jk}/{\hat \lambda}_k=X_j/{\hat a}_1$ 
does not depend on $k=1,2$.

\quad
Finally,  using  (\ref{2014-10-27+2}), (\ref{2014-10-27+3}) we write
\begin{equation}\nonumber
 \E^{\star} e^{{\bf i}tL_1^{(3)}+{\bf i}sL_2^{(3)}}
-
\E^{\star} e^{{\bf i}tL_1^{(4)}+{\bf i}sL_2^{(4)}}
=
\E^{\star}\bigl(g_1(t)-f_1(t)\bigr)g_2(s)
+
\E^{\star}\bigl(g_2(s)-f_2(s)\bigr)f_1(t)
\end{equation}      
and invoke the bounds  $\E^{\star}\bigl(g_1(t)-f_1(t)\bigr)g_2(s)=o(1)$ and
$\E^{\star}\bigl(g_2(s)-f_2(s)\bigr)f_1(t)=o(1)$,
which are obtained in the same way as relation
(22) in \cite{BloznelisDamarackas2013}. We note that the proof of
 (22) in \cite{BloznelisDamarackas2013} uses the moment conditions
$\E X_1^2<\infty$, $\E Y_1<\infty$ that are assumed to hold in the statement of
our Theorem \ref{T2}. 
\end{proof}


\begin{lem}
  \label{DTV+}
   Assume that conditions of part (ii) of Theorem \ref{T2} are satisfied.
Then for any realized values $Y_1,Y_2,X_1$ and any integers $k_1,k_2\ge 0$
we have
\begin{equation}\label{2014-10-28++0}
 \PP^{\star}(L_1=k_1,L_2=k_2)-\PP^{\star}({\cal U}^{***}_{k_1,k_2})=o(1).
\end{equation}
\end{lem}


\begin{proof}[Proof of Lemma \ref{DTV+}]
Before the proof we introduce some notation.
Let $\varepsilon>0$.
For $2\le i\le m$ denote
\begin{displaymath}
 {\mathbb I}'_i
=
{\mathbb I}'_i(\varepsilon)
=
{\mathbb I}_{\{X_i\beta_n^{-1/2}b_1<\varepsilon\}},
\quad
\
\gamma_i=X_i\beta_n^{-1/2}b_1{\mathbb I}'_i
\quad
\
{\text{and}}
\quad
\
\theta_k=\sum_{2\le i\le m}\lambda_{ik}\gamma_i,
\quad
k=1,2. 
\end{displaymath}
Given $X,Y$, let ${\tilde {\mathbb I}}_{2},\dots, {\tilde {\mathbb I}}_{m}$ be 
independent
Bernoulli random variables with success probabilities
\begin{displaymath}
{\tilde \PP}({\tilde {\mathbb I}}_{i}=1)
=
1-{\tilde \PP}({\tilde {\mathbb I}}_{i}=0)
=
\gamma_i,
\end{displaymath} 
and let ${\bar \xi}_{hi}$, $2\le i\le m$, $h=1,2$, be independent Poisson random
variables  with mean values
\begin{equation}\nonumber
{\tilde \E}{\bar \xi}_{1i}=\sum_{3\le j\le n}\lambda_{ij},
\qquad
{\tilde \E}{\bar \xi}_{2i}=\frac{b_1}{\sqrt{\beta_n}}X_i.
\end{equation}
We assume that given $X,Y$ the sequences
$\{{\mathbb I}_{i}\}_{i=2}^m$, $\{{\tilde {\mathbb I}}_{i}\}_{i=2}^m$,
and $\{{\bar \xi}_{hi}\}_{i=2}^m$, $h=1,2$ are  independent.
Next, we introduce  random vectors 
${\bar {\bf L}}^{(h)}=({\bar L}_1^{(h)}, {\bar L}_2^{(h)})$, $0\le h\le 6$. 
Denote
\begin{displaymath}
 {\bar {\bf L}}^{(0)}=({\bar L}_1^{(0)}, {\bar L}_2^{(0)}):=(L_1,L_2),
\qquad
{\bar {\bf L}}^{(6)}=({\bar L}_1^{(6)}, {\bar L}_2^{(6)}):=(\Lambda_3,\Lambda_4).
\end{displaymath}
For $k=1,2$ and $h=1,2$  denote
\begin{eqnarray}\nonumber
&&
{\bar L}_k^{(h)}=\sum_{2\le i\le m}{\mathbb I}_{ik}{\bar \xi}_{hi},
\qquad
{\bar L}_k^{(3)}=\sum_{2\le i\le m}{\mathbb I}_{ik}{\mathbb I}'_i{\bar \xi}_{2i},
\qquad
{\bar L}_k^{(4)}=\sum_{2\le i\le m}{\mathbb I}_{ik}{\tilde {\mathbb I}}_i.
\end{eqnarray}
Furthermore, given $X,Y$, let ${\bar L}_1^{(5)}$ and  ${\bar L}_2^{(5)}$ 
be 
independent Poisson random variables with mean values 
${\tilde \E}{\bar L}_k^{(5)}=\theta_k$.

We shall show below that for any integers $r_1,r_2\ge 0$ we have
\begin{eqnarray}\label{2014-10-30+1}
 &&
d^{\star}_{TV}({\bf L}^{(0)}, {\bf L}^{(1)})=o(1),
\\
\label{2014-10-30+3}
&&
\PP^{\star}\bigl({\bar {\bf L}}^{(1)}=(r_1,r_2)\bigr)
-
\PP^{\star}\bigl({\bar {\bf L}}^{(2)}=(r_1,r_2)\bigr)=o(1),
\\
\label{2014-10-30+4}
&&
\E^{\star} |L_k^{(2)}-L_k^{(3)}|=o(1),
\qquad
k=1,2,
\\ 
\label{2014-10-30+2}
&&
d^{\star}_{TV}({\bf L}^{(3)}, {\bf L}^{(4)})\le \varepsilon(Y_1+Y_2)a_2b_1,
\\
\label{2014-11-1++1}
&&
\bigl|
\PP^{\star}\bigl({\bar {\bf L}}^{(4)}=(r_1,r_2)\bigr)
-
\PP^{\star}\bigl({\bar {\bf L}}^{(6)}=(r_1,r_2)\bigr)
\bigr|
\le \varepsilon(Y_1+Y_2)a_2b_1+o(1).
\end{eqnarray}
Since $\varepsilon>0$ can be chosen  arbitrarily small, relations
 (\ref{2014-10-30+1}), (\ref{2014-10-30+3}), (\ref{2014-10-30+4}),
(\ref{2014-10-30+2}), and
(\ref{2014-11-1++1})
imply (\ref{2014-10-28++0}).

\smallskip

\quad
Let us prove (\ref{2014-10-30+1}). 
  Introduce random variables
\begin{displaymath}
 {\bar L}_k^{\{t\}}
=
\sum_{2\le i\le t}{\mathbb I}_{ik}u_i
+
\sum_{t+1\le i\le m}{\mathbb I}_{ik}{\bar \xi}_{1i},
\qquad
t=1,\dots, m,
\quad
k=1,2.
\end{displaymath}
Note that 
${\bar L}_k^{\{1\}}={\bar L}_k^{(1)}$ and ${\bar L}_k^{\{m\}}={\bar L}_k^{(0)}$. 
We apply  triangle inequality
\begin{equation}\label{2014-10-30+5}
{\tilde d}_{TV}
\bigl(
({\bar L}_1^{\{m\}},{\bar L}_2^{\{m\}}),({\bar L}_1^{\{1\}},{\bar L}_2^{\{1\}})
\bigr)
\le
\sum_{t=2}^m
{\tilde d}_{TV}
\bigl(
({\bar L}_1^{\{t-1\}},{\bar L}_2^{\{t-1\}}),({\bar L}_1^{\{t\}},{\bar L}_2^{\{t\}})
\bigr)
\end{equation}
and  estimate each summand as follows
\begin{eqnarray}\label{2014-10-30+6}
 {\tilde d}_{TV}
\bigl(
({\bar L}_1^{\{t-1\}},{\bar L}_2^{\{t-1\}}),({\bar L}_1^{\{t\}},{\bar L}_2^{\{t\}})
\bigr)
&\le& 
{\tilde d}_{TV}
\bigl(
({\mathbb I}_{t1}u_t, {\mathbb I}_{t2}u_t),
({\mathbb I}_{t1}{\bar \xi}_{1t},{\mathbb I}_{t2}{\bar \xi}_{1t})
\bigr)
\\
\label{2014-10-30+7}
&\le& 
{\tilde \PP}\bigl(({\mathbb I}_{t1},{\mathbb I}_{t2})\not=(0,0)\bigr)
{\tilde d}_{TV}(u_t,{\bar \xi}_{1t})
\\
\label{2014-10-30+8}
&
\le
&
(\lambda_{t1}+\lambda_{t2})\beta_n^{-1/2}X_t({\hat b}_2n^{-1})^{1/2}.
\end{eqnarray}
In  (\ref{2014-10-30+6}) we use the fact that 
${\bar L}_k^{\{t-1\}}$ and ${\bar L}_k^{\{t\}}$ only differ in the $t$-th summand.
 In (\ref{2014-10-30+7}) we  first estimate the total variation distance
between conditional distributions of 
$({\mathbb I}_{t1}u_t, {\mathbb I}_{t2}u_t)$
 and $({\mathbb I}_{t1}{\bar \xi}_{1t},{\mathbb I}_{t2}{\bar \xi}_{1t})$ 
given ${\mathbb I}_{t1}, {\mathbb I}_{t2}, X,Y$
from above by 
${\mathbb I}_{\{(\eta_{t1},\eta_{t2})\not=(0,0)\}}
{\tilde d}_{TV}(u_t,{\bar \xi}_{1t})$
and then take the expected value ${\tilde \E}$.
 In (\ref{2014-10-30+8}) we estimate 
${\tilde \PP}\bigl(({\mathbb I}_{t1},{\mathbb I}_{t2})\not=(0,0)\bigr)
\le
\lambda_{t1}+\lambda_{t2}$
and  
 invoke the inequality
\begin{displaymath}
{ \tilde d}_{TV}(u_t,{\bar \xi}_{1t})
\le 
\min\bigl\{1,\, \beta_n^{-1}X_t^2{\hat b}_2n^{-1}\bigr\}
\le \bigl(\beta_n^{-1}X_t^2{\hat b}_2n^{-1}\bigr)^{1/2},
\end{displaymath}
cf. (\ref{2014-10-06+1}) above.

\quad
Next, using the identities ${\bar L}_k^{\{1\}}={\bar L}_k^{(1)}$ 
and ${\bar L}_k^{\{m\}}={\bar L}_k^{(0)}$
we obtain from (\ref{2014-10-30+5}), (\ref{2014-10-30+8}) that
\begin{eqnarray}\label{2014-10-30+9}
{\tilde d}_{TV}({\bar {\bf L}}^{(0)}, {\bar {\bf L}}^{(1)})
\le
\sum_{t=2}^m(\lambda_{t1}+\lambda_{t2})
\beta_n^{-1/2}X_t({\hat b}_2n^{-1})^{1/2}
=
(Y_1+Y_2)
{\hat a}_2({\hat b}_2/n)^{1/2}.
\end{eqnarray}
Finally, since $\E X_1^2<\infty$ and $\E Y_1<\infty$ imply ${\hat a}_2=O_P(1)$ and 
${\hat b}_2/n=o_P(1)$, we conclude that 
${\tilde d}_{TV}({\bar {\bf L}}^{(0)}, {\bar {\bf L}}^{(1)})=o_{P^{\star}}(1)$.
Hence, 
\begin{displaymath}
 d^{\star}_{TV}({\bar {\bf L}}^{(0)}, {\bar {\bf L}}^{(1)})
=
\E^{\star}{\tilde d}_{TV}({\bar {\bf L}}^{(0)}, {\bar {\bf L}}^{(1)})
=
\E^{\star}
\min
\bigl\{1,{\tilde d}_{TV}({\bar {\bf L}}^{(0)}, {\bar {\bf L}}^{(1)})\bigr\}
=o(1).
\end{displaymath}



Let us prove (\ref{2014-10-30+3}). We proceed as in the proof 
of (\ref{2014-10-21++2}) for $h=2$ above. Denote
$b=\min\{{\hat b}_1,b_1\}$ and $\delta'={\hat b}_1-b$,  
$\delta''=b_1-b$.
Given $X,Y$, let ${\hat \xi}_i$, ${\hat \Delta}'_i$, ${\hat \Delta}''_i$, 
$2\le i\le m$, be independent Poisson random variables, which are independent
of ${\mathbb I}_2,\dots, {\mathbb I}_m$, and  having mean values
\begin{displaymath}
 {\tilde \E}{\hat \xi}_i=\beta_n^{-1/2}X_ib,
\qquad
{\tilde \E}{\hat \Delta}'_i=\beta_n^{-1/2}X_i\delta',
\qquad
{\tilde \E}{\hat \Delta}''_i=\beta_n^{-1/2}X_i\delta''.
\end{displaymath}
Denote 
\begin{displaymath}
{\hat L}_k^{(1)}=\sum_{2\le i\le m}{\mathbb I}_{ik}({\hat \xi}_i+{\hat\Delta}'_i),
\qquad
 {\hat L}_k^{(2)}=\sum_{2\le i\le m}{\mathbb I}_{ik}({\hat \xi}_i+{\hat\Delta}''_i),
\quad
k=1,2.
\end{displaymath}
We note that ${\hat {\bf L}}^{(h)}:=({\hat L}_1^{(h)},{\hat L}_2^{(h)})$ 
has the same distribution
as ${\bf {\bar L}}^{(h)}$, $h=1,2$. Hence it suffices to show 
(\ref{2014-10-30+3}) for ${\hat {\bf L}}^{(1)}, {\hat {\bf L}}^{(2)}$.
Proceeding as in the proof of (\ref{2014-10-21++2}) 
for $h=2$ above, we reduce the problem
to showing that  
\begin{displaymath}
 {\tilde \E}|{\hat L}_k^{(1)}-{\hat L}_k^{(2)}|
=
\sum_{2\le i\le m}\lambda_{ik}
\beta_n^{-1}(\delta'+\delta'')
=
Y_k{\hat a}_2|{\hat b}_1-b_1|=o(1),
\qquad
k=1,2.
\end{displaymath}


Let us prove (\ref{2014-10-30+4}). We have, for $k=1,2$,
\begin{displaymath}
 {\tilde \E} |L_k^{(2)}-L_k^{(3)}|
=
\sum_{2\le i\le m}
(1-{\mathbb I}'_i)
\bigl({\tilde \E}{\mathbb I}_{ik}\bigr)
\bigl({\tilde \E}{\bar \xi}_{2i}\bigr)
\le
Y_kb_1m^{-1}\sum_{2\le i\le m}
(1-{\mathbb I}'_i)X_i^2.
\end{displaymath}
Taking $\E^{\star}$ expected value we obtain
\begin{displaymath}
{\E}^{\star} |L_k^{(2)}-L_k^{(3)}|
=
{\E}^{\star}
\bigl( {\tilde \E} |L_k^{(2)}-L_k^{(3)}|\bigr)
\le
Y_kb_1\E X_2^2{\mathbb I}_{\{X_2\beta_n^{-1/2}b_1\ge \varepsilon\}}
=o(1).
\end{displaymath}
We note that the expectation on right tends to zero since $\beta_n\to+\infty$.



Let us prove (\ref{2014-10-30+2}). We proceed as in the proof of 
(\ref{2014-10-30+1}) for $h=0$ above.
We have
\begin{equation}\nonumber
 {\tilde d}_{TV}({\bf L}^{(3)}, {\bf L}^{(4)})
\le 
\sum_{2\le i\le m}
{\mathbb I}'_i
{\tilde \PP}\bigl({\mathbb I}_{i1},{\mathbb I}_{i2})\not= (0,0)\bigr)
{\tilde d}_{TV}({\bar \xi}_{2i},{\tilde {\mathbb I}}_{i}).
\end{equation}
Next, we estimate 
${\mathbb I}_i'
{\tilde d}_{TV}({\bar \xi}_{2i}, {\tilde {\mathbb I}}_i)\le \gamma_i^2$, by
Le Cam's inequality (\ref{LeCam}), and invoke the inequality
${\tilde \PP}\bigl({\mathbb I}_{i1},{\mathbb I}_{i2})\not= (0,0)\bigr)
\le 
\lambda_{i1}+\lambda_{i2}$. 
We obtain
\begin{displaymath}
{\tilde d}_{TV}({\bf L}^{(3)}, {\bf L}^{(4)})
\le 
\sum_{2\le i\le m}
{\mathbb I}_i'(\lambda_{i1}+\lambda_{i2})\gamma_i^2
\le 
\varepsilon\sum_{2\le i\le m}
{\mathbb I}_i'(\lambda_{i1}+\lambda_{i2})\gamma_i
\le 
\varepsilon (Y_1+Y_2)b_1{\hat a}_2.
\end{displaymath}
Here we used inequality $\gamma_i^2\le \varepsilon\gamma_i$. It follows now that
\begin{displaymath}
d^{\star}_{TV}({\bf L}^{(3)}, {\bf L}^{(4)})
\le
{\E}^{\star}{\tilde d}_{TV}({\bf L}^{(3)}, {\bf L}^{(4)})
\le 
\varepsilon(Y_1+Y_2)b_1a_2.
\end{displaymath}


Let us prove (\ref{2014-11-1++1}).
For $A\subset W$ denote 
${\tilde S}(A)=\sum_{w_i\in A}{\tilde {\mathbb I}}_i$.
For $k=1,2$ we write ${\bar L}_k^{(4)}={\tilde S}({\bf A}_k)$, where
 ${\bf A}_k$ 
denote the set attributes from $W\setminus\{w_1\}$ that are prescribed to $v_k$,
i.e.,
${\bf A}_k =\{w_i\in W\setminus\{w_1\}:\, w_i\to v_k\}$.
Given $r_1,r_2\ge 0$ introduce events
${\cal D}_k=\{{\tilde S}({\bf A}_k)=r_k\}$, $k=1,2$, and 
${\cal D}=\{|{\bf A}_1\cap{\bf A}_2|\ge 1\}$.

We first show that 
\begin{equation}\label{2014-11-1+2}
 \PP^{\star}({\bar L}_1^{(4)}=r_1, {\bar L}_2^{(4)}=r_2)
=
\E^{\star}
\bigl(
 {\tilde \PP}({\bar L}_1^{(4)}=r_1)
 {\tilde \PP}({\bar L}_2^{(4)}=r_2)
\bigr)
+o(1).
\end{equation}
It is convenient to write (\ref{2014-11-1+2}) in the form
\begin{equation}\nonumber
\PP^{\star}({\cal D}_1\cap{\cal D}_2)
=
\E^{\star}
\bigl(
{\tilde \PP}({\cal D}_1)
{\tilde \PP}({\cal D}_2)
\bigr)
+
o(1).
\end{equation}
Now, we observe  that 
given $X,Y$, and ${\bf A}_1,{\bf A}_2$, satisfying 
${\bf A}_1\cap{\bf A}_2=\emptyset$,
the random variables ${\tilde S}({\bf A}_1)$ and  ${\tilde S}({\bf A}_2)$
are independent. Hence  for  ${\bf A}_1\cap{\bf A}_2=\emptyset$ we have
\begin{displaymath}
{\tilde \PP}\bigl({\cal D}_1\cap {\cal D}_2|{\bf A}_1, {\bf A}_2\bigr)
=
{\tilde \PP}({\cal D}_1|{\bf A}_1)
{\tilde \PP}({\cal D}_2|{\bf A}_2).
\end{displaymath}
This identity implies 
\begin{equation}\label{2014-11-1+1}
 {\tilde \PP}({\cal D}_1\cap {\cal D}_2\cap{\bar {\cal D}})
=
{\tilde \E}
\bigl(
{\tilde \PP}({\cal D}_1\cap {\cal D}_2|{\bf A}_1,{\bf A}_2)
{\bar {\mathbb I}}_{{\cal D}}
\bigr)
=
{\tilde \E}
\bigl({\tilde \PP}({\cal D}_1|{\bf A}_1)
{\tilde \PP}({\cal D}_2|{\bf A}_2)
{\bar {\mathbb I}}_{{\cal D}}
\bigr).
\end{equation}
Next, combining inequalities
\begin{eqnarray}\nonumber
&&
 0
\le 
{\tilde \PP}({\cal D}_1\cap {\cal D}_2)
-{\tilde \PP}({\cal D}_1\cap {\cal D}_2\cap{\bar {\cal D}})
\le 
{\tilde \PP}({\cal D}),
\\
\nonumber
&&
0
\le
{\tilde \E}
\bigl({\tilde \PP}({\cal D}_1|{\bf A}_1)
{\tilde \PP}({\cal D}_2|{\bf A}_2)
\bigr)
-
{\tilde \E}
\bigl({\tilde \PP}({\cal D}_1|{\bf A}_1)
{\tilde \PP}({\cal D}_2|{\bf A}_2)
{\bar {\mathbb I}}_{{\cal D}}
\bigr)
\le
{\tilde \PP}({\cal D})
\end{eqnarray}
with (\ref{2014-11-1+1}) and using the identity 
(which holds, since given $X,Y$, the 
random sets ${\bf A}_1$ and ${\bf A}_2$ are independent)
\begin{displaymath}
 {\tilde \E}
\bigl({\tilde \PP}({\cal D}_1|{\bf A}_1)
{\tilde \PP}({\cal D}_2|{\bf A}_2)
\bigr)
=
\bigl({\tilde \E}
{\tilde \PP}({\cal D}_1|{\bf A}_1)
\bigr)
\bigl({\tilde \E}
{\tilde \PP}({\cal D}_2|{\bf A}_2)
\bigr)
=
{\tilde \PP}({\cal D}_1)
{\tilde \PP}({\cal D}_2)
\end{displaymath}
we obtain that
\begin{displaymath}
\bigl|
\PP^{\star}({\cal D}_1\cap{\cal D}_2)
-
\E^{\star}
\bigl(
{\tilde \PP}({\cal D}_1)
{\tilde \PP}({\cal D}_2)
\bigr)
\bigr|
\le 
\E^{\star}{\tilde \PP}({\cal D}).
\end{displaymath}
It remains to show that  $\E^{\star}{\tilde \PP}({\cal D})=o(1)$.
To this aim we apply 
Markov's inequality
\begin{equation}
 {\tilde \PP}({\cal D})
\le 
\sum_{2\le i\le m}{\tilde \E} {\mathbb I}_{i1}{\mathbb I}_{i2}
\le  
\sum_{2\le i\le m}\E \lambda_{i1}\lambda_{i2}
=n^{-1}Y_1Y_2{\hat a}_2
\end{equation}
and obtain $\E^{\star}({\tilde \PP}({\cal D}))\le n^{-1}Y_1Y_2a_2=o(1)$ thus completing 
the proof of (\ref{2014-11-1+2}).

\smallskip

We secondly show that
\begin{equation}\label{2014-11-1+3}
\E^{\star}
\bigr|
 {\tilde \PP}({\bar L}_1^{(4)}=r_1)
 {\tilde \PP}({\bar L}_2^{(4)}=r_2)
-
{\tilde  \PP}({\bar L}_1^{(5)}=r_1)
{\tilde  \PP}({\bar L}_2^{(5)}=r_2)
\bigr|
\le 
\varepsilon(Y_1+Y_2)b_1a_2.
\end{equation}
Denote, for short, the integrand in (\ref{2014-11-1+3}) by
$|q_1q_2-h_1h_2|$.
Since $0\le q_1,q_2,h_1,h_2\le 1$, we have
$|q_1q_2-h_1h_2|
\le |q_1-h_1|+|q_2-h_2|$. We apply Le Cam's inequality (\ref{LeCam})  and obtain
\begin{equation}\nonumber
|q_k-h_k|
\le 
\sum_{2\le i\le m}{\mathbb I}'_i\lambda_{ik}\gamma_{i}^2
\le 
\varepsilon \sum_{2\le i\le m} {\mathbb I}'_i\lambda_{ik}\gamma_i
\le 
\varepsilon Y_1b_1{\hat a}_2,
\qquad
k=1,2.
\end{equation}
Here we estimate $\gamma_i^2\le \varepsilon\gamma_i$. Clearly,  inequality
$\E^{\star}|q_k-h_k|\le \varepsilon Y_kb_1a_2$ imply (\ref{2014-11-1+3}).

Finally, we show that
\begin{equation}\label{2014-11-1+4}
\E^{\star}
\bigl(
{\tilde \PP}({\bar L}_1^{(5)}=r_1)
{\tilde \PP}({\bar L}_2^{(5)}=r_2)
\bigr)
\to
\PP^{\star}(\Lambda_3=r_1)
\PP^{\star}(\Lambda_4=r_2).
\end{equation}
Here we will use the fact that the almost sure convergence ${\hat a}_2\to a_2$ 
implies the 
 convergence in probability 
\begin{equation}\label{2014-11-1+5}
 (\theta_1, \theta_2)\xrightarrow{P} (\lambda_3,\lambda_4).
\end{equation}
Since for $s,t\ge 0$ 
the function $(s,t)\to e^{-t-s}s^{r_1}t^{r_2}/(r_1!r_2!)$ 
is continuous and bounded, (\ref{2014-11-1+5}) implies the convergence 
of expected values
\begin{equation}\label{2014-11-1+6}
\E^{\star}
\frac{e^{-\theta_1}e^{-\theta_2}\theta_1^{r_1}\theta_2^{r_2}}{r_1!r_2!}
\to
\E^{\star}
\frac{e^{-\lambda_3}e^{-\lambda_4}\lambda_3^{r_1}\lambda_4^{r_2}}{r_1!r_2!}
=
\frac{e^{-\lambda_3}e^{-\lambda_4}\lambda_3^{r_1}\lambda_4^{r_2}}{r_1!r_2!}.
\end{equation}
We observe that the quantities on the left (right) sides of (\ref{2014-11-1+4}) and 
(\ref{2014-11-1+6}) are the same. Hence  (\ref{2014-11-1+6}) implies 
(\ref{2014-11-1+4}).

It remains to prove (\ref{2014-11-1+5}). 
Denote ${\hat \lambda}_{k+2}=Y_kb_1{\hat a}_2$, $k=1,2$.
We have for any $\delta>0$
\begin{equation}\label{2014-11-05+1}
 \PP^{\star}(|\theta_k-{\hat \lambda}_{2+k}|\ge \delta)
\le
\delta^{-1}\E^{\star}|\theta_k-{\hat \lambda}_{2+k}|
\le
\delta^{-1}Y_kb_1\E X_2^2(1-{\mathbb I}'_2)=o(1).
\end{equation}
In the last step we used $\E X_2^2(1-{\mathbb I}'_2)=o(1)$.
Now, (\ref{2014-11-1+5}) follows from 
(\ref{2014-11-05+1}) and the bounds ${\hat \lambda}_{k+2}-\lambda_{k+2}=o_P(1)$,
$k=1,2$, which are simple consequences of the fact that $\E X_1^2<\infty$ implies
${\hat a}_2-a_2=o_P(1)$.
\end{proof}

\section{Appendix B}


\begin{lem}\label{L3} 
 Let $c,\varkappa, h> 0$. Let  $Z,\Lambda_Z$ 
be non-negative random variables such that
\linebreak
$\PP(\Lambda_Z=r)=\E \bigl(e^{-Z}Z^r/r!\bigr)$, $r=0,1,\dots$.

(i) The relation $\PP(Z>t)=(c+o(1))t^{-\varkappa}$ as $t\to+\infty$ implies
\begin{equation}\label{2014-11-7++1} 
 \PP(\Lambda_Z>t)=(c+o(1))t^{-\varkappa}
\qquad
{\text{as}}
\quad
t\to+\infty.
\end{equation}

(ii) If  $Z\in {\cal P}_{c,\varkappa}$ 
then $\PP(\Lambda_Z=r)\approx cr^{-\varkappa}$.

(iii) If $hZ$ is integer valued and satisfies $\PP(hZ=r)\approx c(h/r)^{\varkappa}$ then
$\PP(\Lambda_Z=r)\approx c hr^{-\varkappa}$.
\end{lem}

\begin{proof}[Proof of Lemma \ref{L3}]
 Let us prove (i). 
We first collect auxiliary inequalities.
For a Poisson random variable $\Lambda$ with mean 
$z>0$  and $0<s<z<t$  we have, see \cite{mitzenmacherANDupfal},
\begin{equation}\label{2014-11-8++2}
 \PP(\Lambda\ge t)\le e^{-z}(ez/t)^t,
\qquad
\PP(\Lambda\le s)\le e^{-z}(ez/s)^s.
\end{equation}
For $0<x<1$ and $y>1$ we have
\begin{equation}\label{2014-11-8++1}
\ln (1-x)
\le 
-x-0.5x^2,
\qquad
\ln (1+x)
\le
x-0.25x^2,
\qquad 
\ln(1+y)\le y\ln 2.
\end{equation}

In order to prove (\ref{2014-11-7++1}) we 
split the probability $\PP(\Lambda_Z>t)=P_1+P_2+P_3$, where 
\begin{eqnarray}\nonumber
&&
P_1=\PP(\Lambda_Z>t, Z<t_1),
\quad
P_2=\PP(\Lambda_Z>t, Z\in[t_1,t_2]),
\quad
P_3=\PP(\Lambda_Z>t, Z>t_2),
\\
\nonumber
&&
t_1=t(1-\varepsilon),
\qquad
t_2=t(1+\varepsilon),
\qquad
\varepsilon=t^{-1/3},
 \end{eqnarray}
and show that $P_2=(c+o(1))t^{-\varkappa}$ and $P_k=o(t^{-\varkappa})$, $k=1,2$.

We have 
\begin{displaymath}
P_2
\le
\PP(t_1\le Z\le t_2)
=(c+o(1))(t_1^{-\varkappa}-t_2^{-\varkappa})
=
o(t^{-\varkappa}).  
\end{displaymath}
Next, we estimate $P_1$. Given $z<t_1$ we  denote ${\bar z}=1-z/t$. Using 
(\ref{2014-11-8++2}), (\ref{2014-11-8++1}) we obtain
\begin{displaymath}
\PP(\Lambda_Z>t|Z=z)\le e^{-z}(ez/t)^t
=
e^{t-z+t\ln(1-{\bar z})}
\le
e^{t-z-t{\bar z}-0.5t{\bar z}^2}
=
e^{-0.5t{\bar z}^2}
\le 
e^{-0.5t\varepsilon^2}.
\end{displaymath}
Hence, $P_1=\E \bigl(\PP(\Lambda_Z>t|Z){\mathbb I}_{\{Z<t_1\}}\bigr)
\le 
e^{-0.5\varepsilon^2t}=o(t^{-\varkappa})$.
In order to evaluate $P_3$ we observe that $\PP(Z>t_2)=(c+o(1))t^{-\varkappa}$ 
and write
\begin{displaymath}
 P_3=\PP(Z>t_2)
-
\PP(\Lambda_Z\le t, Z>t_2).
\end{displaymath}
Then we estimate
\begin{displaymath}
\PP(\Lambda_Z\le t, Z>t_2)
=
\E \bigl( \PP(\Lambda_Z\le t|Z){\mathbb I}_{\{Z>t\}}\bigr)
=
o(t^{-\varkappa})
\end{displaymath} 
using the inequality
\begin{equation}\label{2014-11-8++3}
\PP(\Lambda_Z\le t|Z=z)
\le 
\max\bigl\{e^{-0.25\varepsilon^2t},e^{-t\ln (e/2)}\bigr\},
\qquad
\forall 
\
z>t_2.
\end{equation}
It remains  to prove (\ref{2014-11-8++3}).  
 We have, see  (\ref{2014-11-8++2}),
\begin{displaymath}
\PP(\Lambda_Z\le t|Z=z)
\le 
e^{-z}(ez/t)^t
=
e^{t-z+t\ln(1+y)}.
\end{displaymath}
Here $y=zt^{-1}-1$. 
For $0<y<1$ the quantity on the right is less than $e^{-0.25\varepsilon^2t}$,
by the second inequality of (\ref{2014-11-8++1}). For $y\ge 1$,
by the third inequality
of (\ref{2014-11-8++1}),
the quantity on the right is less than
$e^{-(z-t)(1-\ln 2)}\le e^{-t\ln (e/2)}$,
since $z-t=ty\ge t$.

\bigskip

Let us prove (ii). We only consider the case of integer valued $Z$. For an absolute continuous $Z$ proof is the same.
Denote  $\varepsilon_r =r^{-1/2}\ln r$. We split  
\begin{eqnarray}\nonumber
&&
  \PP(\Lambda_Z=r)=\E \bigl(e^{-Z}Z^r/r!\bigr)
=
J_1+J_2+J_3,
\qquad
J_k=\E \bigl({\mathbb I}_{\{Z\in Q_k\}}e^{-Z}Z^r/r!\bigr),
\\
\nonumber
&&
Q_1=[0,r(1-\varepsilon_r)), 
\qquad
Q_2=[r(1-\varepsilon_r), r(1+\varepsilon_r)],
\qquad
Q_3=(r(1+\varepsilon_r), +\infty),
\end{eqnarray}
 and show that
\begin{equation}\label{2014-11-7++3}
 J_1=o(r^{-\varkappa}), 
\qquad
J_3=o(r^{-\varkappa}),
\qquad
J_2=(c+o(1))r^{-\varkappa}
\qquad
{\text{as}}
\quad
r\to+\infty.
\end{equation}

In  the proof we  use the following bounds related to the integral 
representation of
Euler's Gamma function.
Given interval $Q\subset [0,+\infty)$, denote 
\begin{displaymath}
I_Q=\int_Q\frac{e^{-x}x^r}{r!}dx,
\qquad
S_Q= \sum_{j\in Q_k}\frac{e^{-j}j^r}{r!}.
\end{displaymath}
For large $r$ we have
\begin{equation}\label{2014-11-7++4}
 I_{Q_1}\le r^{-0.3\ln r}, 
\qquad
I_{Q_3}\le r^{-0.3\ln r},
\qquad
1-2r^{-0.3\ln r}\le I_{Q_2}\le 1.
\end{equation}
Relations (\ref{2014-11-7++3}) follow from (\ref{2014-11-7++4}) and the 
approximation $S_{Q_k}=(1+o(1))I_{Q_k}$, $k=1,2,3$,
\begin{eqnarray}\nonumber
 J_k
&=&
\sum_{j\in Q_k}\frac{e^{-j}j^r}{r!}\PP(Z=r)
\le
 S_{Q_k}
=
(1+o(1))I_{Q_k},
\qquad
k=1,3,
\\
\nonumber
J_2
&=&
\sum_{j\in Q_2}\frac{e^{-j}j^r}{r!}\PP(Z=r)
=
\frac{c+o(1)}{r^{\varkappa} }S_{Q_2}
=
\frac{c+o(1)}{r^{\varkappa}}(1+o(1))
I_{Q_2}.
\end{eqnarray}

For reader's convenience we provide a proof of (\ref{2014-11-7++4}).
We note that the third relation of (\ref{2014-11-7++4}) follows from the 
first two and the well known fact that 
$I_{[0,+\infty)}=1$. 
In the proof, for sufficiently small $\varepsilon>0$, we apply 
the inequalities, which follow by Taylor's expansion,
\begin{displaymath}
e^{\varepsilon}(1-\varepsilon)\le 1-\varepsilon^2/3,
\qquad
e^{-\varepsilon}(1+\varepsilon)\le 1-\varepsilon^2/3.
\end{displaymath}

Since the function $x\to e^{-x}x^r$
increases on $Q_1$, we have  $I_{Q_1}\le ze^{-z}z^r/r!$, where 
$z=r(1-\varepsilon_r)$ is the right end-point of $Q_1$.
Using Stirling's formula we obtain for large $r$
\begin{displaymath}
 ze^{-z}z^r/r!
=
(1+o(1))
e^{r\varepsilon_r}(1-\varepsilon_r)^{r+1}(r/2\pi)^{1/2}
\le 
r^{1/2}(1-\varepsilon_r^2/3)^r
\le 
r^{-0.3\ln r}.
\end{displaymath}
In the last step we used $1-\varepsilon\le e^{-\varepsilon}$. 
Furthermore, for $y=r(1+\varepsilon_r)$ we have
decreases on $Q_3$ we have 
\begin{equation}\label{2014-11-7++5}
 I_{Q_3}
\le 
\frac{e^{-y}y^{r}}{r!}
\int_{y}^{+\infty}e^{-(x-y)}\Bigl(1+\frac{x-y}{y}\Bigr)^rdx.
\end{equation}
Using Stirling's formula we obtain 
\begin{displaymath}
e^{-y}y^r/r!
=
(1+o(1))
\bigl(e^{-\varepsilon_r}(1+\varepsilon_r)\bigr)^r (2\pi r)^{-1/2}
\le 
r^{-1/2}(1-\varepsilon_r^2/3)^r
\le 
r^{-0.5-0.3\ln r}.
\end{displaymath}
Combining this inequality with the following upper 
bound on the integral of (\ref{2014-11-7++5})
\begin{displaymath}
 \int_0^{+\infty}e^{-t}(1+t/y)^rdt
\le
 \int_0^{+\infty}e^{-t}e^{(rt/y)}dt
=
1+\varepsilon_r^{-1}\le r^{1/2}
\end{displaymath}
we obtain (\ref{2014-11-7++4}). Here we used the inequality 
$(1+t/y)^r\le e^{(rt/y)}$.

We omit the proof of (iii) since it is similar to that of (ii).
\end{proof}

{\it Acknowledgement.} The present study was motivated by a question of
Konstantin Avra\-tchen\-kov  about the  degree-degree distribution 
in a random intersection graph.  A discussion with  Konstantin Avratchenkov
and Jerzy Jaworski about
the influence of the clustering property on the degree-degree distribution 
was the starting point of this research.

\end{document}